\documentclass[10pt]{article}
\pdfoutput=1
\usepackage[ruled,linesnumbered]{algorithm2e}
\usepackage{amscd}
\usepackage{times}
\usepackage{amsmath}
\usepackage{amssymb}
\usepackage{amsthm}
\usepackage{dsfont}
\usepackage{mathrsfs}
\usepackage{relsize}
\usepackage{graphicx,graphics}
\usepackage{pstricks}
\usepackage{pst-all}
\usepackage[body={15cm,23cm}, top=2cm]{geometry}
\usepackage{bm}
\usepackage{multirow}
\usepackage{blkarray}
\usepackage{setspace}
\usepackage{enumerate}
\usepackage[inline,shortlabels]{enumitem}
\usepackage{cases}
\usepackage[linewidth=1pt]{mdframed}
\usepackage[numbers,sort&compress]{natbib}
\usepackage[e]{esvect}
\usepackage{tablefootnote}
\usepackage[colorlinks, citecolor=blue]{hyperref}
\hypersetup{linkcolor=blue,
anchorcolor=blue,
citecolor=blue }

\graphicspath{{fig/}}
\allowdisplaybreaks
\newtheorem{theorem}{Theorem}[section]
\newtheorem{corollary}[theorem]{Corollary}

\newtheorem{proposition}[theorem]{Proposition}

\newtheorem{lemma}[theorem]{Lemma}

\newtheorem{conj}{Conjecture}

\SetKwComment{Comment}{$\triangleright$\ }{}
\newcommand{\RNum}[1]{\uppercase\expandafter{\romannumeral #1\relax}}

\newcommand{\mat}[1]{\mathbf{#1}}

\def \R {\mathbb{R}}

\def \rank {\mathrm{rank}}

\def \st {\mathrm{ s.t. }}

\begin{document}

\title{Rank properties and computational methods for orthogonal tensor decompositions}

\author{Chao Zeng \thanks{E-mail: zengchao@nankai.edu.cn}}

\date{}
\maketitle

\emph{\textbf{Abstract\ } The orthogonal decomposition factorizes a tensor into a sum of an
orthogonal list of rank-one tensors. We present several properties of orthogonal rank. We find
that a subtensor may have a larger orthogonal rank than the whole tensor and prove the
lower semicontinuity of orthogonal rank. The lower
semicontinuity guarantees the existence of low orthogonal rank approximation. To fit the
orthogonal decomposition, we propose an algorithm based on the augmented Lagrangian method and
guarantee the orthogonality by a novel orthogonalization procedure. Numerical experiments show
that the proposed method has a great advantage over the existing methods for strongly orthogonal
decompositions in terms of the approximation error.
}

\vspace{0.5cm}
\noindent \textbf{Keywords\ } orthogonal tensor decomposition, orthogonal rank,
augmented Lagrangian method, orthogonalization

\vspace{0.5cm}
\noindent \textbf{Mathematics Subject Classifications (2010)\ }  15A69, 49M27, 90C26, 90C30

\section{Introduction}

Given a tensor $\mathcal{A}\in \R^{I_1\times \cdots \times I_N}$,  the CANDECOMP/PARAFAC (CP)
decomposition factorizes it into a sum of rank-one tensors:
$$
\mathcal{A} = \sum_{k=1}^{K} \mat{v}^{(1)}_k \otimes \cdots \otimes \mat{v}^{(N)}_k,
$$
where $\mat{v}^{(n)}_k\in \R^{I_n},k=1,\ldots,K,n=1,\ldots,N$. Usually, it is difficult to determine
the number $K$ for expressing $\mathcal{A}$ exactly \cite{haastad1990tensor,hillar2013most}.
Hence, the following approximate CP decomposition is more meaningful in practical applications:
$$
\min_{\mat{v}^{(n)}_r\in \R^{I_n}}\left\|\mathcal{A}-\sum_{r=1}^{R} \mat{v}^{(1)}_r
\otimes \cdots \otimes \mat{v}^{(N)}_r \right\|,
$$
where $R$ is a prescribed number. This problem is just to find a best rank-$R$ approximation
to $\mathcal{A}$. Unfortunately, this problem has no solution in general
\cite{de2008tensor,krijnen2008on}. See the discussion in Section \ref{sec_cp} for details.

As mentioned in \cite{de2008tensor}, the major open question in tensor approximation is how to
overcome the ill-posedness of the low rank approximation problem. One natural
strategy is to impose orthogonality constraints, because the orthogonality is an inherent property
of second-order tensor rank decompositions, i.e., matrix singular value decompositions (SVD).
The orthogonal tensor decomposition can be traced back to \cite{comon1994independent} for the symmetric
case, and then is studied in \cite{kolda2001orthogonal} for the general case:
\begin{equation}\label{intro}
\mathcal{A} = \sum_{r=1}^{R} \mathcal{T}_r, \quad \text{where } \rank(\mathcal{T}_r)=1
\text{ and } \left\langle\mathcal{T}_s,
\mathcal{T}_t \right \rangle = 0 \text{ for all } 1\leq s\neq t \leq R.
\end{equation}
This decomposition is related to nuclear and spectral norms of tensors; see
\cite{derksen2016nuclear,friedland2018nuclear,li2018orthogonal}. In \cite{lim2013blind},
the orthogonality constraint is extended to general angular constraints, where several properties
including the existence, uniqueness and exact recoverability are discussed.
As a special case of decompositions with angular constraints, the orthogonal tensor
decomposition also has these properties.

The earliest method for computing the low orthogonal rank approximation
is the greedy approach presented in \cite{kolda2001orthogonal},
where one rank-one component is updated in one iteration. Specifically, suppose we have
obtained $k$ rank-one components. The $(k+1)$st rank-one component is updated by
\begin{align*}
\min_{\mathcal{U}}\quad & \left\|\mathcal{A}-\sum_{r=1}^{k}\mathcal{T}_r-\mathcal{U}\right\| \\
\st \quad & \rank(\mathcal{U})=1 \text{\quad and \quad } \left\langle\mathcal{T}_r,
\mathcal{U} \right \rangle = 0, \ r=1,\ldots,k.
\end{align*}
This method is reasonable only if the Eckart-Young theorem \cite{eckart1936the} can be extended
to the orthogonal decomposition, i.e., the best low orthogonal rank approximation can be obtained
by truncating the orthogonal rank decomposition (see Section \ref{sec3} for the definition).
Refer to \cite[Section 5]{kolda2001orthogonal} for details.
However, a counterexample presented in \cite{kolda2003counterexample}
shows that such an extension is not possible. Suppose $\mathcal{T}_r=\otimes_{n=1}^N \mat{v}_r^{(n)}$
in (\ref{intro}). The constraint has the following form
$$
\prod_{n=1}^{N}\left\langle \mat{v}_s^{(n)},\mat{v}_t^{(n)}
\right\rangle=0 \quad  \text{ for all } s\neq t.
$$
This means that there exists at least one $m \in \{1,\ldots,N\}$ such that
$\left\langle \mat{v}_s^{(m)},\mat{v}_t^{(m)}\right\rangle=0$. However, we cannot determine
the number $m$ for different pairs of $s,t$. This is the main difficulty in fitting orthogonal
decompositions. Practical existing algorithms are proposed
by fixing the number $m$. Actually, these algorithms are aimed at
strongly orthogonal decompositions, whose one or more factor matrices are orthogonal;
see Section \ref{sec_od} for details. The case where all (normalized) factor matrices
are orthogonal is considered in \cite{chen2008on}; the case where one factor matrix is
orthogonal is considered in \cite{sorensen2012canonical,wang2015orthogonal}; the case
where an arbitrary number of factor matrices are orthogonal is considered in
\cite{guan2019numerical}. In a recent work \cite{yang2020epsilon}, a globally convergent
algorithm is developed to compute general strongly orthogonal decompositions.
All these algorithms follow a similar framework, by
combining the alternating minimization method and the polar decomposition.
For factor matrices with general angular constraints, a proximal gradient algorithm
is proposed in \cite{nazih2020using}. In \cite{martin2008jacobi}, the Jacobi SVD algorithm
is extended to reduce a tensor to a form with the $\ell_2$ norm of the diagonal
vector being maximized. The resulting form is not diagonal and hence this is not
an algorithm for orthogonal decompositions discussed in this paper.

In this paper, we first study orthogonal rank. We find that there are
many differences between orthogonal rank and tensor rank. Orthogonal rank
may be variant under the invertible $n$-mode product, a subtensor may have a larger
orthogonal rank than the whole tensor, and orthogonal rank is
lower semicontinuous. A refined upper bound of orthogonal rank \cite{li2018orthogonal}
is given. As for the algorithm, we employ the augmented Lagrangian method to convert
(\ref{intro}) into an unconstrained problem. Then the unconstrained problem can be
solved by gradient-based optimization methods.
To guarantee the orthogonality of the final result, we develop an orthogonalization
procedure. Numerical experiments show that our method has a great advantage over the
existing methods for strongly orthogonal decompositions in terms of the approximation error.

The rest of this paper is organized as follows. Section \ref{sec2} recalls some preliminary
materials. In Section \ref{sec3}, we present several properties of orthogonal
rank. The algorithm is proposed in Section \ref{sec4}. Experimental results are
given in Section \ref{sec5}. Conclusions are presented in Section \ref{sec6}.

\subsection*{Notation}

We use bold-face lowercase letters ($\mat{a}, \mat{b}, \ldots$) to denote vectors,
bold-face capitals ($\mat{A}, \mat{B}, \ldots$) to denote matrices and calligraphic letters
($\mathcal{A}, \mathcal{B}, \ldots$) to denote tensors. The notations $\mat{I}$ and
$\mat{0}$ denote the identity matrix and the zero matrix of suitable dimensions, respectively.
The $(i_1,i_2,\cdots,i_N)$th element of $\mathcal{A}$ is denoted by $a_{i_1i_2\cdots i_N}$.
The $n$-mode product of a tensor $\mathcal{A}$ by a matrix
$\mat{M}$ is denoted by $\mat{M}\cdot_n \mathcal{A}$.
Following \cite{de2008tensor}, we write $\mat{M}_1\cdot_1 \cdots \mat{M}_N\cdot_N \mathcal{A}$
more concisely as $(\mat{M}_1,\cdots,\mat{M}_N)\cdot \mathcal{A}$.

\section{Preliminaries}\label{sec2}

\subsection{Inner product, angle and orthogonality}

Let $\mathcal{A},\mathcal{B} \in \R^{I_1\times \cdots \times I_N}$. The inner product of $\mathcal{A},
\mathcal{B}$ is defined by
$$
\langle \mathcal{A},\mathcal{B}\rangle := \sum_{i_1=1}^{I_1}\cdots \sum_{i_N=1}^{I_N}
a_{i_1,\cdots,i_N}b_{i_1,\cdots,i_N},
$$
and the norm of $\mathcal{A}$ induced by this inner product is
$\|\mathcal{A}\|= \sqrt{\langle \mathcal{A},\mathcal{A}\rangle}$.
Let $\mathcal{U}=\mat{u}^{(1)}\otimes \cdots \otimes \mat{u}^{(N)}$ and
$\mathcal{V}=\mat{v}^{(1)}\otimes \cdots \otimes \mat{v}^{(N)}$. Then
\begin{equation}\label{inner}
\langle \mathcal{U}, \mathcal{V}\rangle = \prod_{n=1}^{N}\langle \mat{u}^{(n)},\mat{v}^{(n)}\rangle
\text{\quad and \quad } \|\mathcal{U}\|=\prod_{n=1}^{N}\|\mat{u}^{(n)}\|.
\end{equation}
We say that $\mathcal{A}$ is a \emph{unit} tensor if $\|\mathcal{A}\|=1$.

The \emph{angle} between $\mathcal{A},\mathcal{B}$ is defined by
\begin{equation}\label{Tangle}
\angle(\mathcal{A},\mathcal{B}):=\arccos\left\langle \frac{\mathcal{A}}{\|\mathcal{A}\|},
\frac{\mathcal{B}}{\|\mathcal{B}\|} \right\rangle.
\end{equation}
Two tensors $\mathcal{A},\mathcal{B}$ are \emph{orthogonal} ($\mathcal{A} \bot \mathcal{B}$) if
$\langle \mathcal{A},\mathcal{B}\rangle=0$, i.e., $\angle(\mathcal{A},\mathcal{B})= \pi/2$.
In (\ref{inner}), $\mathcal{U}$ and $\mathcal{V}$ are
orthogonal if $\prod_{n=1}^{N}\langle \mat{u}^{(n)},\mat{v}^{(n)}\rangle =0$.
This leads to other options for defining orthogonality of two rank-one tensors. Given
$1 \leq i_1 < \cdots < i_M \leq N$, we say that $\mathcal{U}$ and $\mathcal{V}$ are
\emph{$(i_1,\cdots,i_M)$-orthogonal} if
$$
\left \langle \mat{u}^{(i_m)},\mat{v}^{(i_m)}\right\rangle = 0 \quad
\forall 1 \leq m \leq M.
$$
If $M=N$, we say that $\mathcal{U}$ and $\mathcal{V}$ are \emph{completely orthogonal}.

A list of tensors $\mathcal{T}_1,\cdots,\mathcal{T}_m$ is said to be orthogonal if
$\langle \mathcal{T}_i,\mathcal{T}_j\rangle = 0$ for all distinct $i,j\in \{1,\ldots,m\}$.
An orthogonal list of tensors is an orthonormal list if each of its elements is a unit tensor.
Similarly, we can define an $(i_1,\cdots,i_M)$-orthogonal list of rank-one tensors.

\subsection{CP decompositions and tensor rank}\label{sec_cp}

The CP decomposition factorizes a tensor into a sum of rank-one
tensors:
\begin{equation}\label{cpd}
\mathcal{A} = \sum_{r=1}^{R}\mat{v}_r^{(1)}  \otimes \cdots
\otimes \mat{v}_r^{(N)}:=[\![ \mat{V}^{(1)},\cdots,\mat{V}^{(N)}]\!],
\end{equation}
where the $n$th \emph{factor matrix} is
\begin{equation}\label{factor}
\mat{V}^{(n)} = \begin{bmatrix}\mat{v}_1^{(n)}& \cdots & \mat{v}_R^{(n)} \end{bmatrix}.
\end{equation}
An interesting property of tensors is that their CP decompositions are often unique.
Refer to \cite[Section 3.2]{kolda2009tensor} for detailed introductions.
The most famous results \cite{kruskal1977three,sidiropoulos2000uniqueness} on the uniqueness
condition depend on the concept of $k$-rank. The $k$-rank of a matrix $\mat{M}$, denoted by
$k_{\mat{M}}$, is the largest integer such that every set containing $k_{\mat{M}}$ columns
of $\mat{M}$ is linearly independent. For the CP decomposition (\ref{cpd}), its uniqueness
condition presented in
\cite{sidiropoulos2000uniqueness} is
\begin{equation}\label{unique}
\sum_{n=1}^{N} k_{\mat{V}^{(n)}} \geq 2R+N-1.
\end{equation}

The rank of $\mathcal{A}$ is defined by $\rank(\mathcal{A}):=\min \left\{R: \mathcal{A}
= \sum_{r=1}^{R}\mat{v}_r^{(1)} \otimes \cdots \otimes \mat{v}_r^{(N)}\right\}$.
Given $R>0$, the following problem
\begin{equation}\label{cp}
\min_{\rank(\mathcal{B})\leq R} \left\|\mathcal{A} - \mathcal{B} \right\|
\end{equation}
aims to find the \emph{best rank-$R$ approximation} of $\mathcal{A}$. However, (\ref{cp})
has no solution in general \cite{de2008tensor,krijnen2008on}. The essential reason is the following
feature of tensor rank.

\begin{proposition}[\cite{de2008tensor}]\label{notlow}
Let $R\geq 2$. The set $\{\mathcal{A}\in \R^{I_1\times \cdots \times I_N}: \rank(\mathcal{A})\leq R\}$
is not closed in the normed space $\R^{I_1\times \cdots \times I_N}$. That is,
the function $\rank(\mathcal{A})$ is not lower semicontinuous.
\end{proposition}

\subsection{Orthogonal decompositions}\label{sec_od}

The orthogonal decomposition factorizes a tensor into a sum of an orthogonal list 
of rank-one tensors:
\begin{equation}\label{ord}
\mathcal{A} = \sum_{r=1}^{R} \mathcal{T}_r, \quad \text{where } \rank(\mathcal{T}_r)=1 \text{ and } 
\mathcal{T}_s \bot \mathcal{T}_t \text{ for all } 1\leq s\neq t \leq R.
\end{equation}
The following lemma can be obtained by a direct calculation based on (\ref{inner}).

\begin{lemma}\label{hard}
The decomposition (\ref{cpd}) is an orthogonal decomposition if and only if
$\mat{V}^{(1)^T}\mat{V}^{(1)}\circledast \cdots \circledast\mat{V}^{(N)^T}\mat{V}^{(N)}$
is diagonal, where ``$\circledast$" is the Hadamard product.
\end{lemma}

The $(i_1,\cdots,i_M)$-orthogonal decomposition factorizes a tensor into a sum of an
$(i_1,\cdots,i_M)$-orthogonal list of rank-one tensors. Any type of an
$(i_1,\cdots,i_M)$-orthogonal decomposition is called a \emph{strongly orthogonal decomposition}
\footnote{Strongly orthogonal decomposition has a different definition in \cite{kolda2001orthogonal}.}.
Clearly, a strongly orthogonal decomposition is also an orthogonal decomposition.
However, we are not in general guaranteed that a strongly orthogonal decomposition exists.
Simple examples include the tensors with $\rank(\mathcal{A})>\max\{I_1,\ldots,I_N\}$
\footnote{Such tensors exist. See \cite[Lemma 4.7]{de2008tensor} for an example.}. This is
because an $(i_1,\cdots,i_M)$-orthogonal list consists of at most $\min\{I_{i_1},\ldots,I_{i_M}\}$
elements. Related discussions can be found in \cite{kolda2001orthogonal,chen2008on}.

There is a lot research on strongly orthogonal decompositions. The $(1,\cdots,N)$-orthogonal
decomposition, also called the \emph{completely orthogonal decomposition}, is discussed in
\cite{chen2008on}. The $(n)$-orthogonality, where $1\leq n \leq N$, is considered in
\cite{sorensen2012canonical,wang2015orthogonal}. General strongly orthogonal decompositions
are considered in \cite{guan2019numerical,yang2020epsilon}. General angular (see (\ref{Tangle}))
constraint decompositions are discussed in \cite{lim2013blind}.

\section{Properties of orthogonal rank}\label{sec3}

The \emph{orthogonal rank} of $\mathcal{A}$ is the smallest possible value of $R$ for which a
decomposition (\ref{ord}) is possible.
If $R=\rank_{\bot}(\mathcal{A})$ in (\ref{ord}), then (\ref{ord}) is called
an \emph{orthogonal rank decomposition}.

Clearly, $\rank_{\bot}(\mathcal{A})\geq \rank(\mathcal{A})$. The following lemma gives a sufficient
condition for $\rank_{\bot}(\mathcal{A})> \rank(\mathcal{A})$.

\begin{lemma}\label{r<or}
Let $\mat{V}^{(n)} \in \R^{I_n \times R}$ for $n=1,\ldots N$. If
$\mat{V}^{(1)^T}\mat{V}^{(1)}\circledast \cdots
\circledast\mat{V}^{(N)^T}\mat{V}^{(N)}$ is not diagonal, $R\geq 2$ and
$\rank(\mat{V}^{(n)})=R \ \forall n=1,\ldots N$, then
$\mathcal{A} = [\![ \mat{V}^{(1)},\cdots,\mat{V}^{(N)}]\!]$ satisfies
$\rank(\mathcal{A})=R<\rank_{\bot}(\mathcal{A})$.
\end{lemma}
\begin{proof}
Since $\rank(\mat{V}^{(n)})=R$ and $R\geq 2$, we have
$$
\sum_{n=1}^{N} k_{\mat{V}^{(n)}}=NR \geq 2R+N-1.
$$
By (\ref{unique}), this decomposition is unique and $\rank(\mathcal{A})=R$.

On the other hand, by Lemma \ref{hard}, this decomposition is not an orthogonal decomposition.
Due to the uniqueness, there does not exist an orthogonal decomposition
with $R$ terms, i.e., $\rank_{\bot}(\mathcal{A})>R$.
\end{proof}

In \cite{chen2008on}, the existence of the completely orthogonal decomposition is discussed.
We can use such existence to give a sufficient condition for
$\rank_{\bot}(\mathcal{A})= \rank(\mathcal{A})$.

\begin{lemma}
If $\mathcal{A}$ admits a completely orthogonal decomposition, then
$\rank_{\bot}(\mathcal{A})= \rank(\mathcal{A})$.
\end{lemma}

The proof of this lemma can follow that of Lemma \ref{r<or}. We omit it here.

Suppose $\mathcal{A}$ is a subtensor of $\mathcal{B}$, then
$\rank(\mathcal{A}) \leq \rank(\mathcal{B})$.
It comes as a surprise that the analogue does not hold for orthogonal rank.
See the next proposition.

\begin{proposition}
Let $\mat{V}^{(n)} \in \R^{I_n \times R}$ for $n=1,\ldots N$ and
$\mathcal{A} = [\![ \mat{V}^{(1)},\cdots,\mat{V}^{(N)}]\!]$. If
$\mat{V}^{(1)^T}\mat{V}^{(1)}\circledast \cdots
\circledast\mat{V}^{(N)^T}\mat{V}^{(N)}$ is not diagonal, $R\geq 2$ and
$\rank(\mat{V}^{(n)})=R \ \forall n=1,\ldots N$, then there exists a tensor $\mathcal{B}$ such that
$$
\mathcal{A} \text{ is a subtensor of } \mathcal{B}\quad \text{ and }\quad
\rank_{\bot}(\mathcal{B})<\rank_{\bot}(\mathcal{A}).
$$
\end{proposition}
\begin{proof}
We can find a sufficiently large $t$ such that $t\mat{I}-\mat{V}^{(1)^T}\mat{V}^{(1)}$ is
positive semidefinite. Then there exists a matrix $\mat{M}$ with $R$ columns such that
$$
t\mat{I}-\mat{V}^{(1)^T}\mat{V}^{(1)}=\mat{M}^T\mat{M}.
$$
Denote $\mat{V}=\begin{bmatrix} \mat{V}^{(1)} \\  \mat{M} \end{bmatrix}$. Then
$\mathcal{B}=[\![ \mat{V},\mat{V}^{(2)},\cdots,\mat{V}^{(N)}]\!]$ is an orthogonal decomposition.
Using a proof like that of Lemma \ref{r<or}, we have
$\rank_{\bot}(\mathcal{B})=R<\rank_{\bot}(\mathcal{A})$.
\end{proof}

A basic property of tensor rank is its invariance under the invertible $n$-mode product.
If $\mat{M}_n$ is invertible for $n=1,\dots,N$, \cite[Lemma 2.3]{de2008tensor} tells us that
$$
\rank((\mat{M}_1,\cdots,\mat{M}_N)\cdot \mathcal{A}) = \rank(\mathcal{A}).
$$
However, this property does not hold for orthogonal rank. Counterexamples can be constructed based on
Lemma \ref{r<or}. Due to the fact that $\rank(\mat{V}^{(1)})=R$, there exists an invertible matrix
$\mat{M} \in \R^{I_1\times I_1}$ satisfying $\mat{M}(:,1:R)=\mat{V}^{(1)}$. Then
$\mat{M}^{-1}\mat{V}^{(1)}=\begin{bmatrix} \mat{I} \\ \mat{0} \end{bmatrix}$ and
$\mat{M}^{-1}\cdot_1 \mathcal{A} = [\![ \mat{M}^{-1}\mat{V}^{(1)},\mat{V}^{(2)},\cdots,\mat{V}^{(N)}]\!]$
is an orthogonal decomposition. Therefore,
$$
\rank_{\bot}(\mat{M}^{-1}\cdot_1 \mathcal{A})=\rank(\mat{M}^{-1}\cdot_1 \mathcal{A})
= \rank(\mathcal{A}) < \rank_{\bot}(\mathcal{A}).
$$

If the $n$-mode product is orthogonal, we have the following lemma.

\begin{lemma}\label{orun}
Let $\mathcal{A}\in \R^{I_1\times \cdots \times I_N}$ and $\mat{M}_n \in \R^{I_n \times I_n}$ be
orthogonal for $n=1,\ldots,N$. Then
$$
\rank_{\bot}((\mat{M}_1,\cdots,\mat{M}_N)\cdot \mathcal{A}) = \rank_{\bot}(\mathcal{A}).
$$
\end{lemma}
\begin{proof}
Suppose $\mathcal{A}=[\![ \mat{V}^{(1)},\cdots,\mat{V}^{(N)}]\!]$ is an orthogonal decomposition.
Then $(\mat{M}_1,\cdots,\mat{M}_N)\cdot \mathcal{A}=
[\![ \mat{M}_1\mat{V}^{(1)},\cdots,\mat{M}_N\mat{V}^{(N)}]\!]$ and
$(\mat{M}_1\mat{V}^{(1)})^T\mat{M}_1\mat{V}^{(1)}\circledast \cdots \circledast
(\mat{M}_N\mat{V}^{(N)})^T\mat{M}_N\mat{V}^{(N)}=\mat{V}^{(1)^T}\mat{V}^{(1)}\circledast
\cdots \circledast\mat{V}^{(N)^T}\mat{V}^{(N)}$ is diagonal. Hence,
$\rank_{\bot}((\mat{M}_1,\cdots,\mat{M}_N)\cdot \mathcal{A}) \leq \rank_{\bot}(\mathcal{A})$.

On the other hand, we have
$$
\mathcal{A}=(\mat{M}^T_1,\cdots,\mat{M}^T_N)\cdot[(\mat{M}_1,\cdots,\mat{M}_N)\cdot \mathcal{A})]
$$
and hence $\rank_{\bot}(\mathcal{A})\leq \rank_{\bot}((\mat{M}_1,\cdots,\mat{M}_N)\cdot \mathcal{A})$.
Combining these two parts yields the result.
\end{proof}

In \cite[(2.8)]{li2018orthogonal}, an upper bound of $\rank_{\bot}(\mathcal{A})$ is given as
$$
\rank_{\bot}(\mathcal{A}) \leq \min_{m=1,\ldots,N}\prod_{n\neq m}I_n.
$$
We refine this result in the following proposition.

\begin{proposition}
Let $\mathcal{A}\in \R^{I_1\times \cdots \times I_N}$. Then
$$
\rank_{\bot}(\mathcal{A}) \leq \min_{m=1,\ldots,N}\prod_{n\neq m}\rank_n(\mathcal{A}),
$$
where $\rank_{n}(\mathcal{A})$ is the $n$-rank of $\mathcal{A}$.
\end{proposition}
\begin{proof}
Suppose $\mathcal{A}$ has the following HOSVD \cite{de2000multilinear}:
$$
\mathcal{A} =(\mat{U}_1,\cdots,\mat{U}_N)\cdot \mathcal{S},
$$
where $\mat{U}_n \in \R^{I_n\times I_n}$ is orthogonal and $s_{i_1i_2\cdots i_N}=0$
if there exists a least one $i_n>\rank_n(\mathcal{A})$ for $n=1,\ldots,N$.
It follows from Lemma \ref{orun} that $\rank_{\bot}(\mathcal{A})=\rank_{\bot}(\mathcal{S})$.
Note that
$$
\mathcal{S}=\sum_{i_k,k\neq m} \mat{e}_{i_1}\otimes \cdots \otimes
\mat{e}_{i_{m-1}}\otimes \mathcal{S}(i_1,\ldots,i_{m-1},:,i_{m+1},\ldots,i_N) \otimes
\mat{e}_{i_{m+1}}\otimes \cdots \otimes \mat{e}_{i_N},
$$
where $\mat{e}_{i_k}\in \R^{I_k}$ is the standard basis vector and
$\mathcal{S}(i_1,\ldots,i_{m-1},:,i_{m+1},\ldots,i_N)$ is a mode-$m$ fiber. We can check that this
is an orthogonal decomposition. Hence $\rank_{\bot}(\mathcal{S})$ is less than the
number of all non-zero mode-$m$ fibers, which is at most $\prod_{n\neq m}\rank_n(\mathcal{A})$.
\end{proof}

In contrast to Proposition \ref{notlow}, we have the following proposition for orthogonal rank.

\begin{proposition}\label{close}
For any $R>0$, the set $\{\mathcal{A}\in \R^{I_1\times \cdots \times I_N}:
\rank_{\bot}(\mathcal{A})\leq R\}$ is closed in the normed space $\R^{I_1\times \cdots \times I_N}$.
That is, the function $\rank_{\bot}(\mathcal{A})$ is lower semicontinuous.
\end{proposition}
\begin{proof}
Suppose $\mathcal{A}_m \rightarrow \mathcal{A}$, where $\rank_{\bot}(\mathcal{A}_m ) \leq R$. Then
we can write
$$
\mathcal{A}_m = \sum_{r=1}^{R} \sigma_{r,m} \mathcal{U}_{r,m} \quad \text{with} \quad
\mathcal{U}_{r,m}= \mat{u}^{(1)}_{r,m}\otimes \cdots \otimes
\mat{u}^{(N)}_{r,m},
$$
where $\left\langle \mathcal{U}_{s,m}, \mathcal{U}_{t,m}\right\rangle=0$ for all $s\neq t$ and
$\|\mat{u}^{(n)}_{r,m}\|=1$ for all $n=1,\ldots,N$ and $r=1,\ldots,R$. Then
$$
\sum_{r=1}^{R} \sigma_{r,m}^2 = \|\mathcal{A}_m\|^2.
$$
Since $\|\mathcal{A}_m\| \rightarrow \|\mathcal{A}\|$, $\sigma_{r,m}$ are uniformly bounded. Thus
we can find a subsequence with convergence $\sigma_{r,m_k}\rightarrow \sigma_{r},
\mat{u}^{(n)}_{r,m_k}\rightarrow \mat{u}^{(n)}_{r}$ for all $r$ and $n$. Moreover,
$\lim_{m_k\rightarrow \infty}\left\langle \mathcal{U}_{s,m_k}, \mathcal{U}_{t,m_k}\right\rangle=
\left\langle \mathcal{U}_{s}, \mathcal{U}_{t}\right\rangle = 0$ for all $s\neq t$.
Then
$$
\mathcal{A} = \sum_{r=1}^{R} \sigma_{r}\ \mat{u}^{(1)}_{r}\otimes \cdots \otimes
\mat{u}^{(N)}_{r},
$$
satisfying $\rank_{\bot}(\mathcal{A})\leq R$.
\end{proof}

\section{Algorithms for low orthogonal rank approximation}\label{sec4}

Given $R>0$, finding the \emph{best orthogonal rank-$R$ approximation} of $\mathcal{A}$
is
\begin{equation}\label{ocp}
\min_{\rank_{\bot}(\mathcal{B})\leq R} \left\|\mathcal{A} - \mathcal{B} \right\|.
\end{equation}
By Proposition \ref{close}, we know that the solution of (\ref{ocp}) always exists.
Problem (\ref{ocp}) can be formulated as
\begin{equation}\label{orank}
\begin{aligned}
\min_{\mat{v} \in \R^{P}} \quad & \mathscr{F}(\mat{v}):=\frac{1}{2}\left\|\mathcal{A} -
\sum_{r=1}^{R}\otimes_{n=1}^N \mat{v}_r^{(n)}\right\|^2 \\
\st \quad &  \prod_{n=1}^{N}\left\langle \mat{v}_s^{(n)},\mat{v}_t^{(n)}
\right\rangle=0 \quad  \text{ for all } s\neq t,
\end{aligned}
\end{equation}
where $\mat{v}:=\left[\mat{v}_{1}^{(1)^T}\cdots \mat{v}_{R}^{(1)^T}\cdots \mat{v}_{1}^{(N)^T}
\cdots \mat{v}_{R}^{(N)^T} \right]^T$ and $P=R\sum_{n=1}^{N}I_n$.

We employ the augmented Lagrangian method to solve (\ref{orank}).
The augmented Lagrangian function is
\begin{equation}\label{penalty}
\begin{split}
\mathscr{L}(\mat{v},\bm{\lambda};\mat{c}) := & \mathscr{F}(\mat{v}) + \frac{1}{2} \sum_{s=1}^{R}
\sum_{t=1,t\neq s}^{R}\lambda_{st} \prod_{n=1}^{N}\left\langle \mat{v}_s^{(n)},
\mat{v}_t^{(n)}\right\rangle \\
     &  + \frac{1}{4} \sum_{s=1}^{R}
\sum_{t=1,t\neq s}^{R}c_{st} \prod_{n=1}^{N}\left\langle \mat{v}_s^{(n)},
\mat{v}_t^{(n)}\right\rangle^2,
\end{split}
\end{equation}
where $\lambda_{st}=\lambda_{ts}$ are Lagrange multipliers, $c_{st}=c_{ts}>0$ are penalty parameters
and $\bm{\lambda}=\{\lambda_{st}\},\mat{c}=\{c_{st}\}$. Following \cite[p. 124]{bertsekas2014constrained}
and \cite[Chapter 10.4]{sun2010optimization}, we use a different penalty parameter for each constraint,
which will be specified later.

For each iteration of the augmented Lagrangian method, we need to solve the following problem
\begin{equation}\label{alm}
\min_{\mat{v} \in \R^{P}} \mathscr{L}(\mat{v},\bm{\lambda};\mat{c})
\end{equation}
with $\bm{\lambda},\mat{c}$ given. If $\bm{\lambda}=\{0\},\mat{c}=\{0\}$, (\ref{alm}) is
just (\ref{cp}). Since (\ref{cp}) has no solution in general, the first issue that we need to
make sure is whether (\ref{alm}) has a solution. We have the following proposition.

\begin{proposition}
If $c_{st}>0$ for all $s\neq t$, then (\ref{alm}) always has a solution.
\end{proposition}
\begin{proof}
For convenience, define $\mathscr{E}(\mat{v})=\mathscr{L}(\mat{v},\bm{\lambda};\mat{c})$.
Denote $\mathcal{T}_r = \otimes_{n=1}^N \mat{v}_r^{(n)}$. Then
$$
\mathscr{E}(\mat{v}) = \frac{1}{2}\left\|\mathcal{A} - \sum_{r=1}^{R}\mathcal{T}_r\right\|^2 + \frac{1}{4}
\sum_{s=1}^{R}\sum_{t=1,t\neq s}^{R}c_{st} \left(\left\langle \mathcal{T}_s,\mathcal{T}_t \right\rangle
+\frac{\lambda_{st}}{c_{st}}\right)^2 -
\frac{1}{4}\sum_{s=1}^{R}\sum_{t=1,t\neq s}^{R}\frac{\lambda_{st}^2}{c_{st}}.
$$
Note that
\begin{equation}\label{rearrange}
\otimes_{n=1}^N \mat{v}_r^{(n)} = \otimes_{n=1}^N b_n\mat{v}_r^{(n)}
\quad \text{ when } \prod_{n=1}^N b_n = 1.
\end{equation}
We can scale each $\mat{v}_r^{(n)}$ such that $\|\mat{v}_r^{(n)}\|= \|\mathcal{T}_r\|^{1/N},
n=1,\ldots,N$. Define the following set
$$
W = \{\mat{v} \in \R^{P}: \|\mat{v}_r^{(m)}\| = \|\mat{v}_r^{(n)}\|, 1\leq m, n \leq N,
1\leq r \leq R\}.
$$
The continuity of $\|\cdot\|$ implies that $W$ is closed. We have
$$
\{\mathscr{E}(\mat{v}): \mat{v} \in \R^{P}\} = \{\mathscr{E}(\mat{v}): \mat{v} \in W\}.
$$
Hence, it suffices to show that (\ref{alm}) has a solution on $W$.

Denote $\alpha=\frac{1}{4}\sum_{s=1}^{R}\sum_{t=1,t\neq s}^{R}\frac{\lambda_{st}^2}{c_{st}},
\beta =\min\{c_{st}\},\gamma=\sum_{s=1}^{R}\sum_{t=1,t\neq s}^{R}\frac{|\lambda_{st}|}{c_{st}}$.
For any $\xi \geq \inf \mathscr{E} \geq 0$, if $\mathscr{E} \leq \xi$, then
$\left\|\mathcal{A} - \sum_{r=1}^{R}\mathcal{T}_r\right\| \leq \sqrt{2(\xi+\alpha)}$ and
\begin{align*}
& \sum_{s=1}^{R}\sum_{t=1,t\neq s}^{R}
\left|\left\langle \mathcal{T}_s,\mathcal{T}_t \right\rangle \right| - \gamma
\leq \sum_{s=1}^{R}\sum_{t=1,t\neq s}^{R}\left|\left\langle \mathcal{T}_s,\mathcal{T}_t \right\rangle
+\frac{\lambda_{st}}{c_{st}}\right| \\
\leq  & \sqrt{R(R-1)\sum_{s=1}^{R}\sum_{t=1,t\neq s}^{R}
\left(\left\langle \mathcal{T}_s,\mathcal{T}_t \right\rangle
+\frac{\lambda_{st}}{c_{st}}\right)^2 }
\leq  \sqrt{ \frac{4R(R-1)(\xi+\alpha)}{\beta} } \\
\Longrightarrow \quad & \sum_{s=1}^{R}\sum_{t=1,t\neq s}^{R}
\left|\left\langle \mathcal{T}_s,\mathcal{T}_t \right\rangle \right|
\leq \gamma + \sqrt{ \frac{4R(R-1)(\xi+\alpha)}{\beta} }.
\end{align*}
Hence $\|\sum_{r=1}^{R}\mathcal{T}_r\|\leq \|\mathcal{A} -
\sum_{r=1}^{R}\mathcal{T}_r\| + \|\mathcal{A}\| \leq \sqrt{2(\xi+\alpha)} + \|\mathcal{A}\|$.
For any $\mat{v}\in W$, it follows that
\begin{align*}
& (\sqrt{2(\xi+\alpha)} + \|\mathcal{A}\|)^2 \geq \left\|\sum_{r=1}^{R}\mathcal{T}_r \right\|^2
= \sum_{r=1}^{R}\|\mathcal{T}_r\|^2 + \sum_{s=1}^{R}\sum_{t=1,t\neq s}^{R}
\left\langle \mathcal{T}_s,\mathcal{T}_t \right\rangle  \nonumber \\
\geq & \sum_{r=1}^{R}\|\mathcal{T}_r\|^2 - \sum_{s=1}^{R}\sum_{t=1,t\neq s}^{R}
\left|\left\langle \mathcal{T}_s,\mathcal{T}_t \right\rangle \right|\geq
\sum_{r=1}^{R}\|\mathcal{T}_r\|^2 - \sqrt{ \frac{4R(R-1)(\xi+\alpha)}{\beta} }-\gamma  \nonumber \\
\Longrightarrow \quad & \|\mat{v}_r^{n}\|^2= \|\mathcal{T}_r\|^{2/N} \leq \left((\sqrt{2(\xi+\alpha)}
+ \|\mathcal{A}\|)^2+\sqrt{ \frac{4R(R-1)(\xi+\alpha)}{\beta}}+\gamma\right)^{1/N}. \label{bound}
\end{align*}
That is, the level set $\{\mat{v}\in W: \mathscr{E}(\mat{v}) \leq \xi,\xi \geq \inf \mathscr{E} \}$
is bounded. Combining with the fact that $\mathscr{E}(\mat{v})$ is continuous and $W$ is closed,
it follows from \cite[Theorem 1.9]{rockafellar2009variational} that
$\mathscr{E}$ can attain its minimum on $W$.
\end{proof}

The gradient of the objective function with respect to $\mat{v}$ has a very good structure. The
calculation of the gradient of the first term of $\mathscr{L}$ can be found in
\cite[Theorem 4.1]{acar2011a}. Note that $c_{st}=c_{ts},\left\langle \mat{v}_s^{(n)},
\mat{v}_t^{(n)} \right\rangle = \left\langle \mat{v}_t^{(n)},\mat{v}_s^{(n)} \right\rangle$. Direct
calculation gives the following lemma.

\begin{lemma}
The partial derivatives of the objective function $\mathscr{L}$ in (\ref{penalty}) are given by
$$
\frac{\partial \mathscr{L}}{\partial \mat{v}_r^{(n)}} = -\mat{w}_r^{(n)} + \sum_{s=1}^{R}\gamma_{sr}^{(n)}
\mat{v}_s^{(n)} + \sum_{s=1,s\neq r}^R \left(\lambda_{sr}\gamma^{(n)}_{sr}+c_{sr}
\gamma_{sr}^{(n)^2}\left\langle \mat{v}_s^{(n)},
\mat{v}_r^{(n)}\right\rangle \right)\mat{v}_s^{(n)},
$$
where $\mat{w}_r^{(n)}= \left({\mat{v}_r^{(1)}}^T,\cdots,{\mat{v}_r^{(n-1)}}^T,\mat{I},
{\mat{v}_r^{(n+1)}}^T,\cdots,{\mat{v}_r^{(N)}}^T \right)\cdot \mathcal{A}$ and
$\gamma_{sr}^{(n)}=\prod_{m=1,m\neq n}^{N}\left\langle \mat{v}_s^{(m)},
\mat{v}_r^{(m)}\right\rangle$.
\end{lemma}

With the relationship introduced in \cite[Section 2.6]{kolda2009tensor}, $\mat{w}_r^{(n)}$
can be rewritten as
$$
\mat{w}_r^{(n)}=\mat{A}_{(n)} \left(\mat{v}_r^{(N)}\circledcirc \cdots \circledcirc \mat{v}_r^{(n+1)}
\circledcirc \mat{v}_r^{(n-1)}\circledcirc \cdots \circledcirc \mat{v}_r^{(1)} \right),
$$
where $\mat{A}_{(n)}$ is the mode-$n$ unfolding of $\mathcal{A}$
and $``\circledcirc"$ is the Kronecker product. Denote
$$
\bm{\Gamma}^{(n)}=\mat{V}^{(1)^T}\mat{V}^{(1)}\circledast\cdots \circledast
\mat{V}^{(n-1)^T}\mat{V}^{(n-1)}\circledast \mat{V}^{(n+1)^T}\mat{V}^{(n+1)}
\circledast \cdots \circledast \mat{V}^{(N)^T}\mat{V}^{(N)},
$$
where $\mat{V}^{(n)}$ is defined in (\ref{factor}).
We can observe that $\gamma_{st}^{(n)}= \bm{\Gamma}^{(n)}(s,t)$. Define
matrices $\bm{\Lambda},\mat{C} \in \R^{R \times R}$ by
\begin{equation}\label{C}
\bm{\Lambda}(i,j) = \begin{cases} \lambda_{ij}, & \mbox{if } i\neq j \\ 0,
& \mbox{otherwise},\end{cases} \quad
\mat{C}(i,j) = \begin{cases} c_{ij}, & \mbox{if } i\neq j \\ 0, & \mbox{otherwise},\end{cases}
\end{equation}
and denote
$$
\mat{V}^{(-n)} = \mat{V}^{(N)}\odot \cdots \odot \mat{V}^{(n+1)}
\odot \mat{V}^{(n-1)}\odot \cdots \odot \mat{V}^{(1)},
$$
where ``$\odot$" is the Khatri-Rao product.
Then, we can rewrite the gradient in matrix form, as the following corollary shows.

\begin{corollary}\label{gra}
The partial derivatives of the objective function $\mathscr{L}$ in (\ref{penalty}) satisfy
$$
\left[\frac{\partial \mathscr{L}}{\partial \mat{v}_1^{(n)}} \cdots
\frac{\partial \mathscr{L}}{\partial \mat{v}_R^{(n)}}\right] = - \mat{A}_{(n)} \mat{V}^{(-n)}
+ \mat{V}^{(n)}\left(\bm{\Gamma}^{(n)}+\bm{\Gamma}^{(n)} \circledast
\bm{\Lambda} + \bm{\Gamma}^{(n)}\circledast \bm{\Gamma}^{(n)}\circledast \mat{V}^{(n)^T}
\mat{V}^{(n)} \circledast \mat{C}\right).
$$
\end{corollary}

\subsection{Algorithm: OD-ALM}

Suppose we have obtained the solution $\mat{v}_{[k]}$ for the $k$th iteration. Now we introduce how to
solve $\mat{v}_{[k+1]}$ for the $(k+1)$st iteration.

We use $\mat{v}_{[k]}$ as the initialization of the $(k+1)$st iteration.
By (\ref{rearrange}), we scale the initialization such that $\|\mat{v}^{(m)}_{r,[k]}\|=
\left(\prod_{n=1}^{N}\|\mat{v}^{(n)}_{r,[k]}\|\right)^{1/N},m=1,\ldots,N$. This scaling can avoid
the situation that some $\|\mat{v}^{(n_1)}_{r,[k]}\|$ is too big and some
$\|\mat{v}^{(n_2)}_{r,[k]}\|$ is too small, where $1\leq n_1,n_2\leq N$.

Note that the solution of each iteration does not satisfy the constraint of (\ref{orank}) exactly.
The effect of the penalty terms of (\ref{penalty}) is just to make
$|\langle \mathcal{T}_s,\mathcal{T}_t \rangle |$ as small as possible, where
$\mathcal{T}_r = \otimes_{n=1}^N \mat{v}_r^{(n)}$ for all $r=1,\ldots,R$.
By (\ref{Tangle}), we have
$$
|\langle \mathcal{T}_s,\mathcal{T}_t \rangle | = \|\mathcal{T}_s\|\|\mathcal{T}_t\|
\left|\cos \angle(\mathcal{T}_s,\mathcal{T}_t)\right|.
$$
Hence, a small value of $|\langle \mathcal{T}_s,\mathcal{T}_t \rangle |$
cannot result in $\angle(\mathcal{T}_s,\mathcal{T}_t)$ being close to $\pi/2$ directly.
To avoid the influence of the norms $\|\mathcal{T}_r\|$, an ideal strategy is to
replace (\ref{penalty}) by the following function
\begin{align*}
\mathscr{L}'(\mat{v},\bm{\lambda};\mat{c})= & \mathscr{F}(\mat{v}) + \frac{1}{2} \sum_{s=1}^{R}
\sum_{t=1,t\neq s}^{R}\lambda_{st} \prod_{n=1}^{N}\left\langle \frac{\mat{v}_s^{(n)}}{\|\mat{v}_s^{(n)}\|},
\frac{\mat{v}_t^{(n)}}{\|\mat{v}_t^{(n)}\|} \right\rangle \\
     &  + \frac{\mu}{4} \sum_{s=1}^{R}
\sum_{t=1,t\neq s}^{R} \prod_{n=1}^{N}\left\langle \frac{\mat{v}_s^{(n)}}{\|\mat{v}_s^{(n)}\|},
\frac{\mat{v}_t^{(n)}}{\|\mat{v}_t^{(n)}\|} \right\rangle^2.
\end{align*}
However, this would make the subproblem rather difficult to solve. We can realize this idea by
setting different penalty parameters for (\ref{alm}):
\begin{equation}\label{c}
c_{st,[k]} =\frac{\mu_{[k]}}{\prod_{n=1}^N \|\mat{v}_{s,[k]}^{(n)}\|^2
\prod_{n=1}^N \|\mat{v}_{t,[k]}^{(n)}\|^2},
\end{equation}
where $\mu_{[k]}>0$. In the matrix form (\ref{C}), the non-diagonal entries of
$\mat{C}_{[k]}$ are the same as those of  $\mu_{[k]} \mat{h}_{[k]}^T\mat{h}_{[k]}$, where
$$\mat{h}_{[k]} = \begin{bmatrix}\frac{1}{\prod_{n=1}^N \|\mat{v}_{1,[k]}^{(n)}\|^2} &
\cdots & \frac{1}{\prod_{n=1}^N \|\mat{v}_{R,[k]}^{(n)}\|^2}   \end{bmatrix} \in \R^{1\times R}.
$$
Then $\mat{v}_{[k+1]}$ can be obtained by solving $\min_{\mat{v} \in \R^{P}}
\mathscr{L}(\mat{v},\bm{\lambda}_{[k]};\mat{c}_{[k]})$.

At last, the Lagrange multiplier $\lambda_{st,[k+1]}$ is updated by
$\lambda_{st,[k+1]}= \lambda_{st,[k]}+
c_{st,[k]}\prod_{n=1}^{N}\left\langle \mat{v}_{s,[k+1]}^{(n)},\mat{v}_{t,[k+1]}^{(n)}\right\rangle$,
whose matrix form is
\begin{equation}\label{updatelambda}
\bm{\Lambda}_{[k+1]} = \bm{\Lambda}_{[k]} + \mat{C}_{[k]}\circledast \left(
\circledast_{n=1}^N \mat{V}_{[k+1]}^{(n)^T}\mat{V}_{[k+1]}^{(n)} \right).
\end{equation}

Now we introduce how to develop a systematic scheme for the augmented Lagrangian method. The standard
procedure of the augmented Lagrangian method tells us that we need to increase the penalty parameters
gradually to a sufficiently large value. This procedure is rather important for (\ref{alm}),
because $\mathscr{L}$ is nonconvex. The later subproblems corresponding to larger penalty parameters
can be solved relatively efficiently by warm starting from the previous solutions. By (\ref{c}),
we need to set $\mu_{[k+1]}$ sufficiently large such that
\begin{equation}\label{increase}
c_{st,[k+1]} > c_{st,[k]}.
\end{equation}
Usually, we can avoid checking this condition by simply setting a sufficiently large gap between
$\mu_{[k+1]}$ and $\mu_{[k]}$. The whole procedure of the augmented Lagrangian method is presented in
Algorithm \ref{alg1}. Here we choose $\mu_{[k+1]}=10\mu_{[k]}$, which has a good performance
for the numerical examples. In practical applications, $\{\mu_{[k]}\}$ can be chosen flexibly and
adaptively.

\begin{algorithm}[!ht]
\caption{Orthogonal Decomposition by Augmented Lagrangian Method (OD-ALM)}
\label{alg1}
\small
\SetAlgoLined
\DontPrintSemicolon
\SetKwFunction{SVD}{SVD}
\SetKwFunction{Up}{Update-N}
\KwIn{Tensor $\mathcal{A}$, number of components $R$, initialization
$\mat{v}_{[0]}$; $\bm{\Lambda}_{[0]}=\bm{0},\mu_{[0]}=1$; $k=0$}
\KwOut{Approximate solution $\mat{v}_{[k]}$ of the orthogonal rank-$R$ approximation
to $\mathcal{A}$}
\Repeat{}{
\For{$r = 1, \dots, R$}{
$\delta_r \leftarrow \prod_{n=1}^{N}\|\mat{v}^{(n)}_{r,[k]}\|$\Comment*[r]{Compute the norm of
$\otimes_{n=1}^N \mat{v}_{r,[k]}^{(n)}$}
}
\For{$r = 1, \dots, R$}{
\For{$n = 1, \dots, N$}{
$\mat{v}^{(n)}_{r,[k]} \leftarrow \frac{\delta_r^{1/N}}{\|\mat{v}^{(n)}_{r,[k]}\|}\mat{v}^{(n)}_{r,[k]}$
\Comment*[r]{scale the initialization}
}}
$\mat{h} \leftarrow \begin{bmatrix}1/\delta_1^2 & \cdots & 1/\delta_R^2\end{bmatrix}$\;
$\mat{C}_{[k]} \leftarrow \mu\ \mat{h}^T\mat{h}$\;
$\mat{C}_{[k]}(i,i) \leftarrow 0 \quad \forall i =1,\ldots,R$\;
$\mat{v}_{[k+1]} \leftarrow \arg\min \mathscr{L}(\mat{v},\bm{\lambda}_{[k]};\mat{c}_{[k]})$
by gradient-based optimization methods with starting point $\mat{v}_{[k]}$, where the gradient is
computed by Corollary \ref{gra}\;
Update $\bm{\Lambda}_{[k+1]}$ by (\ref{updatelambda})\;
$\mu_{[k+1]} \leftarrow 10 \mu_{[k]}$\;
$k\leftarrow k+1$
}( termination criteria met)
\end{algorithm}

The convergence analysis of augmented Lagrangian methods can be found in many textbooks.
See \cite{nocedal2006numerical,bertsekas2014constrained,sun2010optimization} for reference.
Here we extend \cite[Theorem 10.4.2]{sun2010optimization}, which is useful for designing
the termination criteria.

\begin{proposition}
Suppose that (\ref{increase}) holds for Algorithm \ref{alg1}. Then we have
$$
\lim_{k\rightarrow \infty}
\prod_{n=1}^{N}\left\langle \frac{\mat{v}_{s,[k+1]}^{(n)}}{\|\mat{v}_{s,[k]}^{(n)}\|},
\frac{\mat{v}_{t,[k+1]}^{(n)}}{\|\mat{v}_{t,[k]}^{(n)}\|} \right\rangle
 =0 \quad \text{for all}\quad  1\leq s \neq t \leq R.
$$
\end{proposition}
\begin{proof}
We have
\begin{align*}
& \sum_{s\neq t} \frac{\lambda^2_{st,[k+1]}}{c_{st,[k+1]}}
\leq \sum_{s\neq t} \frac{\lambda^2_{st,[k+1]}}{c_{st,[k]}}  \\
= & \sum_{s\neq t}\frac{ \left(\lambda_{st,[k]}+
c_{st,[k]}\prod_{n}\left\langle \mat{v}_{s,[k+1]}^{(n)},\mat{v}_{t,[k+1]}^{(n)}
\right\rangle \right)^2}{c_{st,[k]}} \\
=& \sum_{s\neq t} \frac{\lambda^2_{st,[k]}}{c_{st,[k]}}+
4\left(\mathscr{L}(\mat{v}_{[k+1]},\bm{\lambda}_{[k]};\mat{c}_{[k]})-
\mathscr{F}(\mat{v}_{[k+1]}) \right) \\
\leq & \sum_{s\neq t} \frac{\lambda^2_{st,[k]}}{c_{st,[k]}}+
4\mathscr{L}(\mat{v}_{[k+1]},\bm{\lambda}_{[k]};\mat{c}_{[k]}).
\end{align*}
For any feasible point $\bar{\mat{v}}$ of (\ref{orank}), noting that
$\mathscr{L}(\mat{v}_{[k+1]},\bm{\lambda}_{[k]};\mat{c}_{[k]})\leq
\mathscr{L}(\bar{\mat{v}},\bm{\lambda}_{[k]};\mat{c}_{[k]})=\mathscr{F}(\bar{\mat{v}})$, we have
\begin{align*}
& \sum_{s\neq t} \frac{\lambda^2_{st,[k+1]}}{c_{st,[k+1]}}\leq
\sum_{s\neq t} \frac{\lambda^2_{st,[k]}}{c_{st,[k]}}+
4\mathscr{L}(\mat{v}_{[k+1]},\bm{\lambda}_{[k]};\mat{c}_{[k]}) \\
\leq &  \sum_{s\neq t} \frac{\lambda^2_{st,[k]}}{c_{st,[k]}} + 4\mathscr{F}(\bar{\mat{v}}).
\end{align*}
This suggests that there exists $\delta>0$ such that
$\sum_{s\neq t} \frac{\lambda^2_{st,[k]}}{c_{st,[k]}}\leq \delta k$.
Denote by \\
$d_{st,[k]}:=\lambda_{st,[k]}\prod_{n} \|\mat{v}_{s,[k]}^{(n)}\|
\prod_{n} \|\mat{v}_{t,[k]}^{(n)}\|$. It follows from (\ref{c}) that
 $\sum_{s\neq t} \frac{d^2_{st,[k]}}{\mu_{[k]}}
= \sum_{s\neq t} \frac{\lambda^2_{st,[k]}}{c_{st,[k]}}\leq \delta k$.
By the algorithm, $\mu_{[k]}=10^k$. Hence, $\frac{d_{st,[k]}}{\mu_{[k]}}=o(1)$.

For any feasible point $\bar{\mat{v}}$ of (\ref{orank}), we have
\begin{align*}
& \mathscr{F}(\bar{\mat{v}})=\mathscr{L}(\bar{\mat{v}},\bm{\lambda}_{[k]};\mat{c}_{[k]})
\geq \mathscr{L}(\mat{v}_{[k+1]},\bm{\lambda}_{[k]};\mat{c}_{[k]}) \\
= & \mathscr{F}(\mat{v}_{[k+1]}) + \frac{1}{2}\sum_{s\neq t}d_{st,[k]}\prod_{n=1}^{N}\left\langle
\frac{\mat{v}_{s,[k+1]}^{(n)}}{\|\mat{v}_{s,[k]}^{(n)}\|},
\frac{\mat{v}_{t,[k+1]}^{(n)}}{\|\mat{v}_{t,[k]}^{(n)}\|} \right\rangle \\
&  + \frac{1}{4} \sum_{s\neq t}\mu_{[k]}
\prod_{n=1}^{N}\left\langle \frac{\mat{v}_{s,[k+1]}^{(n)}}
{\|\mat{v}_{s,[k]}^{(n)}\|},\frac{\mat{v}_{t,[k+1]}^{(n)}}{\|\mat{v}_{t,[k]}^{(n)}\|} \right\rangle^2 \\
= & \mathscr{F}(\mat{v}_{[k+1]}) + \frac{1}{4}\sum_{s\neq t}
\mu_{[k]}\left[\left(\prod_{n=1}^{N}\left\langle \frac{\mat{v}_{s,[k+1]}^{(n)}}
{\|\mat{v}_{s,[k]}^{(n)}\|},\frac{\mat{v}_{t,[k+1]}^{(n)}}{\|\mat{v}_{t,[k]}^{(n)}\|} \right\rangle
+ \frac{d_{st,[k]}}{\mu_{[k]}}\right)^2 - \left(\frac{d_{st,[k]}}{\mu_{[k]}}\right)^2 \right] \\
\geq & \frac{1}{4}\sum_{s\neq t}
\mu_{[k]}\left[\left(\prod_{n=1}^{N}\left\langle \frac{\mat{v}_{s,[k+1]}^{(n)}}
{\|\mat{v}_{s,[k]}^{(n)}\|},\frac{\mat{v}_{t,[k+1]}^{(n)}}{\|\mat{v}_{t,[k]}^{(n)}\|} \right\rangle
+ o(1)\right)^2 - o(1) \right].
\end{align*}
Noting that $\lim_{k\rightarrow \infty}\mu_{[k]}=\infty$ and $\mathscr{F}(\bar{\mat{v}})$
is bounded, we obtain the result.
\end{proof}

\begin{corollary}\label{threshold}
Suppose that (\ref{increase}) holds for Algorithm \ref{alg1}, and
$\prod_n\frac{\|\mat{v}_{r,[k]}^{(n)}\|}{\|\mat{v}_{r,[k+1]}^{(n)}\|}$ is bounded for all
$r$ and $k$. Then we have
$$
\lim_{k\rightarrow \infty}
\prod_{n=1}^{N}\left\langle \frac{\mat{v}_{s,[k]}^{(n)}}{\|\mat{v}_{s,[k]}^{(n)}\|},
\frac{\mat{v}_{t,[k]}^{(n)}}{\|\mat{v}_{t,[k]}^{(n)}\|} \right\rangle
 =0 \quad \text{for all}\quad  1\leq s \neq t \leq R.
$$
\end{corollary}

\subsection{Orthogonalization of rank-one tensors}

OD-ALM can only obtain an approximate solution of (\ref{orank}). We need to develop an
orthogonalization procedure to make the orthogonality constraint exact for the final result.

Suppose we have obtained a decomposition by OD-ALM:
$$
\mathcal{A} \approx \sum_{r=1}^{R} \otimes_{n=1}^N \mat{v}_r^{(n)}.
$$
First, we normalize each $\mat{v}_r^{(n)}$ to $\mat{u}_r^{(n)}$, i.e.,
$\mat{u}_r^{(n)}=\mat{v}_r^{(n)}/\|\mat{v}^{(n)}_{r}\|$.
Assume that we have orthogonalizated the first $\ell-1$ rank-one components:
$$
\left\langle \otimes_{n=1}^N \mat{u}_s^{(n)}, \otimes_{n=1}^N \mat{u}_t^{(n)}
\right\rangle =0, \quad 1 \leq s \neq t \leq \ell-1.
$$
We start to handle the $\ell$th rank-one component. Denote
$$
\bar{\mat{U}}^{(n)} = \begin{bmatrix}\mat{u}_1^{(n)}& \cdots &\mat{u}_{\ell-1}^{(n)}
\end{bmatrix}, \quad n=1,\ldots,N.
$$
Compute the absolute value of the inner product $\left|\left\langle \mat{u}_r^{(n)},
\mat{u}_{\ell}^{(n)}\right\rangle \right|$ for $n=1,\ldots,N$ and $r=1\ldots, \ell-1$ and stack
the results as a matrix:
$$
\mat{P} = \left|\begin{bmatrix} \mat{u}_{\ell}^{(1)^T}\bar{\mat{U}}^{(1)} \\ \vdots  \\
\mat{u}_{\ell}^{(N)^T}\bar{\mat{U}}^{(N)} \end{bmatrix} \right| \in \R^{N \times (\ell-1)},
$$
where $|\mat{M}|$ denotes the entrywise absolute value of $\mat{M}$. Let $\mat{P}(m_r,r)=
\min \{\mat{P}(1,r),\ldots,\mat{P}(N,r)\}$. That is, $\mat{u}_{r}^{(m_r)}$
and $\mat{u}_{\ell}^{(m_r)}$ are a pair of vectors that is the closest to orthogonality.
Let $\{r: m_r = n\}$ be $\{r_1\ldots, r_{\rho(n)}\}$. We will modify $\mat{u}_{\ell}^{(n)}$
to $\mat{u}_{\ell}^{(n)} - \sum_{j=1}^{\rho(n)}x_j \mat{u}_{r_j}^{(n)}$
such that
$$
\left\langle \mat{u}_{\ell}^{(n)} - \sum_{j=1}^{\rho(n)}x_j \mat{u}_{r_j}^{(n)},
\mat{u}_{s}^{(n)}  \right \rangle =0, \quad s = r_1,\ldots, r_{\rho(n)},
$$
whose matrix form is
$$
\begin{bmatrix}\mat{u}_{r_1}^{(n)} & \cdots & \mat{u}_{r_{\rho(n)}}^{(n)}\end{bmatrix}^T
\begin{bmatrix}\mat{u}_{r_1}^{(n)} & \cdots & \mat{u}_{r_{\rho(n)}}^{(n)}\end{bmatrix}
\begin{bmatrix}x_{1} \\ \vdots \\ x_{\rho(n)} \end{bmatrix} =
\begin{bmatrix}\mat{u}_{r_1}^{(n)} & \cdots & \mat{u}_{r_{\rho(n)}}^{(n)}\end{bmatrix}^T
\mat{u}_{\ell}^{(n)}.
$$
We present the whole procedure of the orthogonalization in Algorithm \ref{alg2}. This procedure
can also be used for generating general orthonormal lists of rank-one tensors.

\begin{algorithm}[!ht]
\caption{Orthogonalization of rank-one tensors}
\label{alg2}
\small
\SetAlgoLined
\DontPrintSemicolon
\SetKwFunction{SVD}{SVD}
\SetKwFunction{Up}{Update-N}
\KwIn{A list of rank-one tensors $\{\mat{v}^{(n)}_{r}\}_{n,r}$}
\KwOut{An orthonormal list rank-one tensors $\{\mat{u}^{(n)}_{r}\}_{n,r}$}
\For{$r = 1, \dots, R$}{
\For{$n = 1, \dots, N$}{
$\eta \leftarrow \|\mat{v}^{(n)}_r\|$\;
$\mat{u}^{(n)}_r \leftarrow \mat{v}^{(n)}_r/\eta$\;
}}
\For{$\ell = 2, \dots, R$}{
\For{$n = 1, \dots, N$}{
$\mat{U} \leftarrow\begin{bmatrix}\mat{u}_1^{(n)} & \cdots & \mat{u}_{\ell-1}^{(n)}
\end{bmatrix}$\;
$\mat{P}(n,:)\leftarrow \left|\mat{u}_{\ell}^{(n)^T}\mat{U}\right|$
}
\For{$r = 1, \dots, \ell-1$}{
Find $\mat{P}(m_r,r)=\min \{\mat{P}(1,r),\ldots,\mat{P}(N,r)\}$}
\For{$n = 1, \dots, N$}{
$\{r_1\ldots, r_{\rho(n)}\} \leftarrow $ all indices satisfying $m_{r_j}=n, j =1,\ldots,\rho(n)$\;
\eIf{$\rho(n)=0$}{$\mat{u}^{(n)}_{\ell} \leftarrow \mat{u}^{(n)}_{\ell}$}
{
$\mat{B} \leftarrow \begin{bmatrix}\mat{u}_{r_1}^{(n)} & \cdots & \mat{u}_{r_{\rho(n)}}^{(n)}\end{bmatrix}$\;
Solve $\mat{B}^T \mat{B} \mat{x} = \mat{B}^T \mat{u}_{\ell}^{(n)}$ for $\mat{x}$\;
$\mat{u}^{(n)}_{\ell} \leftarrow \mat{u}^{(n)}_{\ell} - \mat{B} \mat{x}$\;
$\eta \leftarrow \|\mat{u}^{(n)}_r\|$\;
$\mat{u}^{(n)}_r \leftarrow \mat{u}^{(n)}_r/\eta$\;}
}}
\end{algorithm}

The final orthogonal rank-$R$ approximation is the orthogonal projection of $\mathcal{A}$
onto the space spanned by the orthonormal list $\{\otimes_{n=1}^{N} \mat{u}_1^{(n)},\ldots,
\otimes_{n=1}^{N} \mat{u}_R^{(n)}\}$:
$$
\sum_{r=1}^{R} \sigma_r \otimes_{n=1}^{N} \mat{u}_r^{(n)},
$$
where the coefficient $\sigma_r = \left\langle \mathcal{A},
\otimes_{n=1}^{N} \mat{u}_r^{(n)}\right\rangle$.

\section{Numerical experiments}\label{sec5}

We will show the performance of OD-ALM combined with the orthogonalization procedure in this section.
All experiments are performed on MATLAB R2016a with Tensor Toolbox, version 3.0
\cite{TTB_Software} on a laptop (Intel Core i5-6300HQ CPU @  2.30GHz, 8.00G RAM).
The test data include both synthetic and real-world tensors. The synthetic tensors
are generated from known ground truth and thus make the evaluation reliable.
Choosing real-world tensors is to assess the approximation ability of orthogonal
decompositions in practice.

The test tensors are shown in Table \ref{data}, where $\mathcal{A}_1,\ldots,\mathcal{A}_4$ are
synthetic tensors and $\mathcal{A}_5,\ldots,\mathcal{A}_8$ are real-world tensors. The tensor
$\mathcal{A}_1$ is a randomly generated tensor, $\mathcal{A}_2$ is a randomly generated rank-5
tensor, and $\mathcal{A}_3$ is a Hilbert tensor also used in \cite{guan2019numerical}.
For $\mathcal{A}_4$, we generate an orthonormal list of rank-one tensors by Algorithm \ref{alg2}
and then use this list to generate an orthogonal rank-5 tensor $\mathcal{B}_1$. The final tensor
$\mathcal{A}_4$ is
$$
\mathcal{A}_4 = \mathcal{B}_1 + \rho \mathcal{B}_2,
$$
where the Gaussian noise tensor $\mathcal{B}_2$ has normally distributed elements, and
$\rho = 0.1 \|\mathcal{B}_1\|/\|\mathcal{B}_2\|$. The tensors $\mathcal{A}_5,\mathcal{A}_6$
are hyperspectral images
\footnote{The hyperspectral image data have been used in \cite{zhu2014spectral} and available at
\url{t https://rslab.ut.ac.ir/data}}, and $\mathcal{A}_7,
\mathcal{A}_8$ are video tensors \footnote{The video data are from the video trace
library \cite{seeling2011video} and available at
\url{http://trace.eas.asu.edu/yuv/}}. We will factorize each tensor into $R$ terms by
different methods, where $R$ is prescribed in Table \ref{data}.

\begin{table}[!ht]
\footnotesize
  \centering
  \caption{The test tensors. The value $R$ is the number of components for all methods.}
  \label{data}
  {\renewcommand{\arraystretch}{1.2}
    \begin{tabular}{cccc}
    \hline
    Tensor  & Size  & $R$ & Note  \\
    \hline
    $\mathcal{A}_1$  & $20\times 16\times 10\times 32$ & 5 & random tensor \\
    $\mathcal{A}_2$  & $20\times 16\times 10\times 32$ & 5 &  rank-5 tensor  \\
    $\mathcal{A}_3$  & $20\times 16\times 10\times 32$ & 5 & $\mathcal{A}_3(i_1,i_2,i_3,i_4)=
    1/(i_1+i_2+i_3+i_4-3)$ \\
    $\mathcal{A}_4$  & $20\times 16\times 10\times 32$ & 5 &  orthogonal rank-5 tensor
    with Gaussian noise \\
    $\mathcal{A}_5$  & $95 \times 95 \times 156$ & 5 &  hyperspectral image -- Samson  \\
    $\mathcal{A}_6$  & $100 \times 100 \times 224$ & 5 &  hyperspectral image -- Jasper Ridge  \\
    $\mathcal{A}_7$  & $144 \times 176 \times 3 \times 300$ & 2 & video data -- Akiyo  \\
    $\mathcal{A}_8$  & $144 \times 176 \times 3 \times 300$ & 2 & video data -- Hall Monitor  \\
    \hline
    \end{tabular}
    }
\end{table}

Suppose $\mathcal{B}$ is an approximation of $\mathcal{A}$ obtained by any method. We use the
relative error (RErr) to evaluate the result:
$$
\text{RErr} = \frac{\|\mathcal{A}-\mathcal{B}\|}{\|\mathcal{A}\|}.
$$

\subsection{Implementation details of OD-ALM}

The initialization is crucial for OD-ALM. We adopt the result of the alternating least squares
algorithm (CP-ALS) \cite{harshman1970foundations,carroll1970analysis,kolda2009tensor}
for (\ref{cp}) as the initialization, because this result is just the numerical solution of
(\ref{alm}) with Lagrange multipliers and penalty parameters equal to zero, which is relatively
near to the solution of the first subproblem of OD-ALM generally. The CP-ALS is with the truncated
HOSVD initialization, and terminates if the relative change in the function value is less than
$10^{-6}$ or the number of iterations exceeds 500.

We have tried the steepest descent method, the conjugate gradient method, the
Broyden-Fletcher-Goldfarb-Shanno (BFGS) method and the limited-memory BFGS (L-BFGS) method
to solve the subproblems (\ref{alm}) and find that the L-BFGS method outperforms the other
three ones. Hence, we use the L-BFGS method with $m=20$ levels of memory in all tests.
We stop the procedure of the L-BFGS method if the relative change between successive iterates
is less than $10^{-8}$, or the $\ell_2$ norm of the
gradient divided by the number of entries is less than $\epsilon_{\text{inner}}$, which
will be specified later. The maximum number of inner iterations is set to be 500.
We adopt the Mor\'{e}-Thuente line search \cite{more1994line} from MINPACK
\footnote{A Matlab implementation, adapted by Dianne P. O'Leary, is available at
\url{http://www.cs.umd.edu/users/oleary/software/}}. For all experiments, Mor\'{e}-Thuente line search
parameters used are as follows: $10^{-4}$ for the function value tolerance, $10^{-2}$ for the gradient
norm tolerance, a starting search step length of 1 and a maximum of 20 iterations.

For the solution $\mat{v}_{[k]}$ of the $k$th subproblem, denote
\begin{equation}\label{angle}
\theta_{[k]} := \max_{s\neq t}\min_{n}\left|\left\langle \frac{\mat{v}_{s,[k]}^{(n)}}
{\|\mat{v}_{s,[k]}^{(n)}\|}, \frac{\mat{v}_{t,[k]}^{(n)}}
{\|\mat{v}_{t,[k]}^{(n)}\|} \right\rangle \right|.
\end{equation}
By Corollary \ref{threshold}, we can terminate the outer iteration when $\theta_{[k]}<
\epsilon_{\text{outer}}$, which will be specified later.
The maximum number of outer iterations is set to be 25.

\subsection{Influence of stopping tolerances}

We test different settings of tolerances: $\epsilon_{\text{inner}}=10^{-3},10^{-4},10^{-5}$
and $\epsilon_{\text{outer}}=10^{-3},10^{-4},10^{-5}$. We record the number of outer iterations
(denoted by ``iter"), and then orthogonalizate the result by Algorithm \ref{alg2}. The whole
running time is recorded, measured in seconds. Finally, we compute the relative error.
The results are shown in Table \ref{table_stop}, which are averaged over 10 times repeated
running.

\begin{table}[!ht]
\footnotesize
  \centering
  \renewcommand\tabcolsep{3.0pt}
  \caption{Results of OD-ALM under different stopping tolerances.}
  \label{table_stop}
  {\renewcommand{\arraystretch}{1.2}
    \begin{tabular}{|c|c|c|cccccccc|}
    \hline
      & $\epsilon_{\text{outer}}$ & $\epsilon_{\text{inner}}$ &  $\mathcal{A}_1$ & $\mathcal{A}_2$
      & $\mathcal{A}_3$   &  $\mathcal{A}_4$ & $\mathcal{A}_5$ & $\mathcal{A}_6$ & $\mathcal{A}_7$
      & $\mathcal{A}_8$    \\
    \hline
    \hline
    \multirow{9}{*}{iter} & \multirow{3}{*}{$10^{-3}$} & $10^{-3}$ & 10 & 10 & 9 & 11
    & 8  &  8   &  9 &  6 \\
    &  & $10^{-4}$ & 10 & 10 & 6 &  10 & 8  & 8  & 9 & 6 \\
    &  & $10^{-5}$ & 10 & 10 & 6 &  8 & 8 & 8 & 9 & 6 \\
    \cline{2-11}
    & \multirow{3}{*}{$10^{-4}$} & $10^{-3}$ & 11 & 11 & 10 & 12 & 9 & 9 & 9 & 7 \\
    &  & $10^{-4}$ & 11 & 11 & 7 &  11 & 8 & 9 & 9  &  7 \\
    &  & $10^{-5}$ & 11 & 11 & 6 & 9 & 8 & 9 & 9 & 7 \\
    \cline{2-11}
    & \multirow{3}{*}{$10^{-5}$} & $10^{-3}$ & 11 & 11 & 11 & 12 & 9 & 11 & 9 & 7 \\
    &  & $10^{-4}$ & 12 & 11 & 7 &  11 & 9 & 9 & 9 & 7 \\
    &  & $10^{-5}$ & 11 & 11 & 9 & 11 & 9 & 9 & 9 & 7 \\
    \hline
    \multirow{9}{*}{time} & \multirow{3}{*}{$10^{-3}$} & $10^{-3}$ & 1.1 & 1.0 & 1.3 & 0.6
    & 4.8 & 13.2 & 15.8 &  15.3 \\
    &  & $10^{-4}$ & 2.6 & 1.3 & 3.2 &  0.5 & 13.9 & 24.4 & 19.6 & 21.8 \\
    &  & $10^{-5}$ & 4.7 & 1.8 & 4.6 & 0.4 & 24.5 & 34.6 & 22.9 & 30.0 \\
    \cline{2-11}
    & \multirow{3}{*}{$10^{-4}$} & $10^{-3}$ & 1.2 & 1.1 & 1.4 & 0.7 & 5.4 & 15.0 & 15.8 & 16.2 \\
    &  & $10^{-4}$ & 2.7 & 1.6 & 3.3 &  0.5 & 14.9 & 25.3 & 19.6 & 23.7 \\
    &  & $10^{-5}$ & 4.9 & 2.6 & 4.6 & 0.5 & 24.3 & 43.1 & 22.9 & 33.7 \\
    \cline{2-11}
    & \multirow{3}{*}{$10^{-5}$} & $10^{-3}$ & 1.2 & 1.1 & 1.6 & 0.7 & 4.9 & 15.7 & 15.6 & 16.2 \\
    &  & $10^{-4}$ & 2.7 & 1.6 & 2.9 & 0.6 & 15.3 & 26.1 & 19.6 & 23.8 \\
    &  & $10^{-5}$ & 4.9 & 3.9 & 5.8 & 0.9 & 24.5 & 41.9 & 23.3 & 33.6 \\
    \hline
    \multirow{9}{*}{RErr} & \multirow{3}{*}{$10^{-3}$} & $10^{-3}$ & 0.9954 & 0.0559 & 0.0640
    & 0.0994 & 0.1831 & 0.2379 & 0.2931 & 0.2278 \\
    &  & $10^{-4}$ & 0.9954 & 0.0559 & 0.0267 &  0.0994 & 0.1831 & 0.2378 & 0.2931 & 0.2278 \\
    &  & $10^{-5}$ & 0.9954 & 0.0559 & 0.0245 &  0.0993 & 0.1831 & 0.2378 & 0.2931 & 0.2278 \\
    \cline{2-11}
    & \multirow{3}{*}{$10^{-4}$} & $10^{-3}$ & 0.9954 & 0.0559 & 0.0640 & 0.0994
    & 0.1831 & 0.2379 & 0.2931 &  0.2278 \\
    &  & $10^{-4}$ & 0.9954 & 0.0559 & 0.0227 &  0.0994 & 0.1831 & 0.2378 & 0.2931 & 0.2278 \\
    &  & $10^{-5}$ & 0.9954 & 0.0559 & 0.0245 &  0.0993 & 0.1831 & 0.2378 & 0.2931 & 0.2278 \\
    \cline{2-11}
    & \multirow{3}{*}{$10^{-5}$} & $10^{-3}$ & 0.9954 & 0.0559 & 0.0640 & 0.0994 & 0.1831
    & 0.2379 & 0.2931 & 0.2278 \\
    &  & $10^{-4}$ & 0.9954 & 0.0559 & 0.0227 &  0.0994 & 0.1831 & 0.2378 & 0.2931 & 0.2278 \\
    &  & $10^{-5}$ & 0.9954 & 0.0559 & 0.0245 & 0.0993 & 0.1831 & 0.2378 & 0.2931 & 0.2278 \\
    \hline
    \end{tabular}
    }
\end{table}

From Table \ref{table_stop}, we can find that OD-ALM has a good performance on convergence:
the outer iteration numbers are at most 12 on average for all cases. The running time would
increase if we choose a smaller tolerance, but there is no improvement on the relative
error for almost all cases. Therefore, we do not recommend using a too small tolerance in practical
applications. We will use $\epsilon_{\text{inner}}=10^{-4},\epsilon_{\text{outer}}=10^{-4}$
for synthetic tensors and $\epsilon_{\text{inner}}=10^{-3},\epsilon_{\text{outer}}=10^{-3}$
for real-world tensors in all remaining tests.

\subsection{Convergence behaviour}

We show the value of $\theta_{[k]}$ defined in (\ref{angle}), the relative change between successive outer
iterates $\|\mat{v}_{[k]} - \mat{v}_{[k-1]}\|/\|\mat{v}_{[k-1]}\|$ and the number of inner iterations
corresponding to each outer iteration in Figure \ref{1-4} and Figure \ref{5-8}.

\begin{figure}[!ht]
\scriptsize
\begin{center}
\begin{tabular}{c@{\hskip 0.2cm}c@{\hskip 0.2cm}c}
\includegraphics[width=3.5cm]{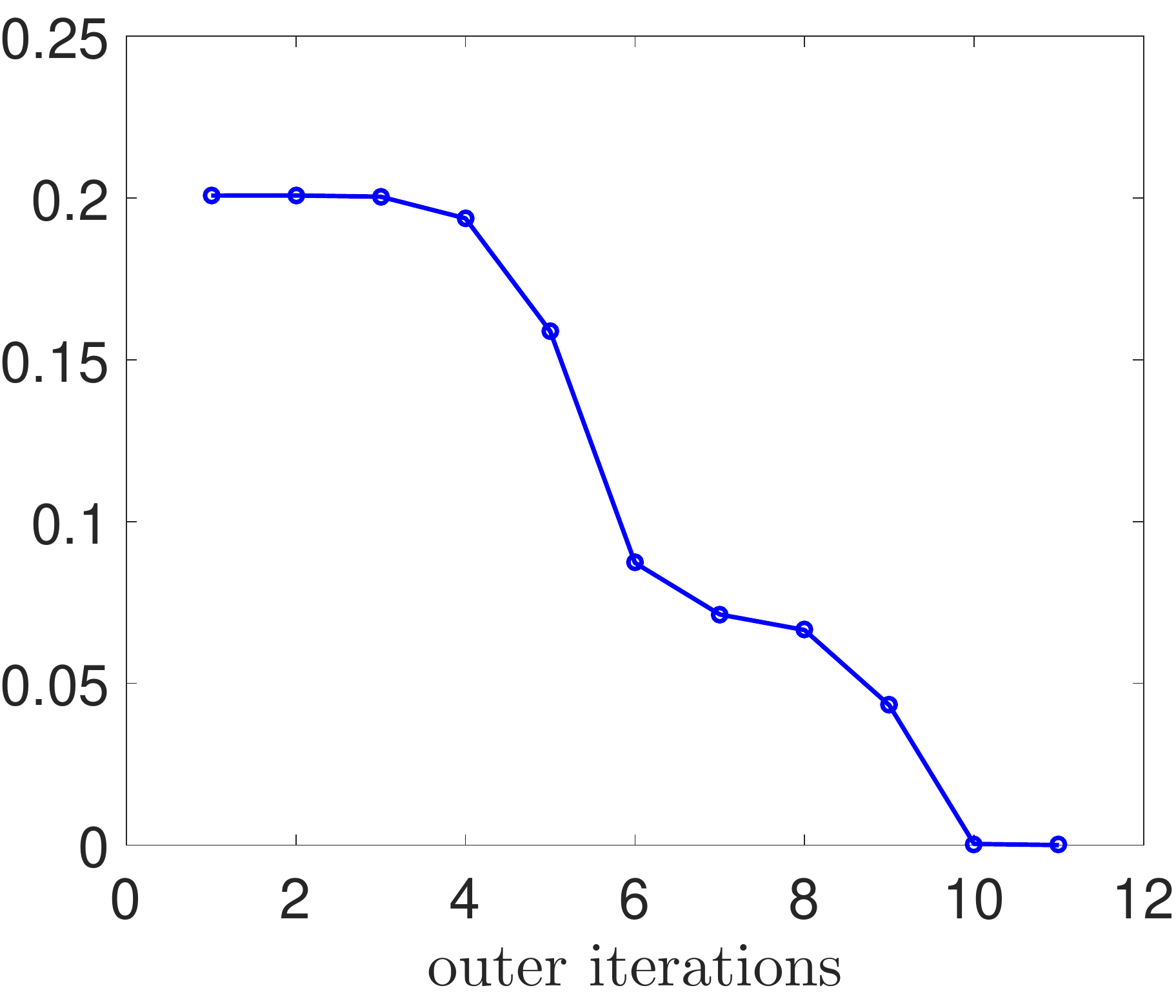} &
\includegraphics[width=3.5cm]{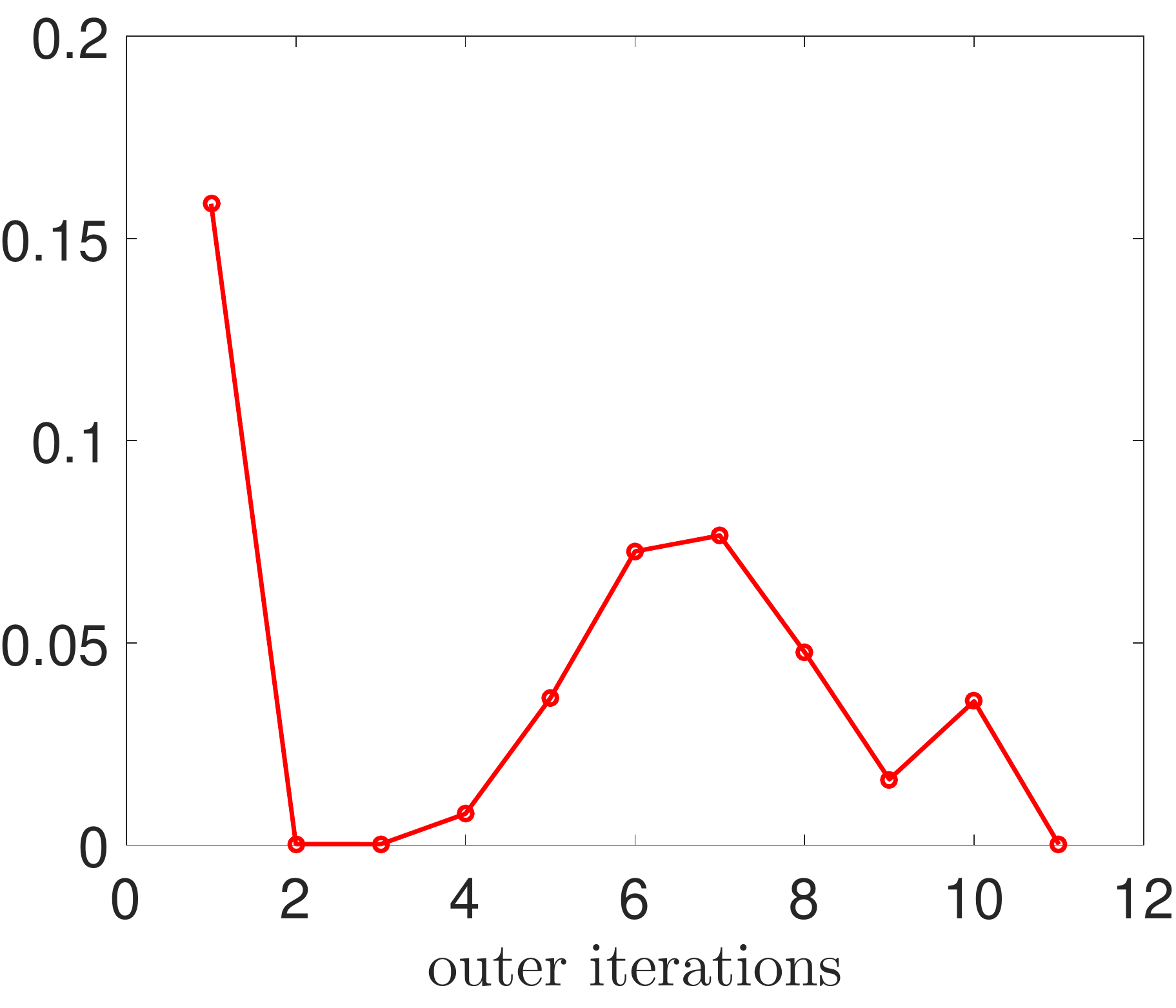} &
\includegraphics[width=3.5cm]{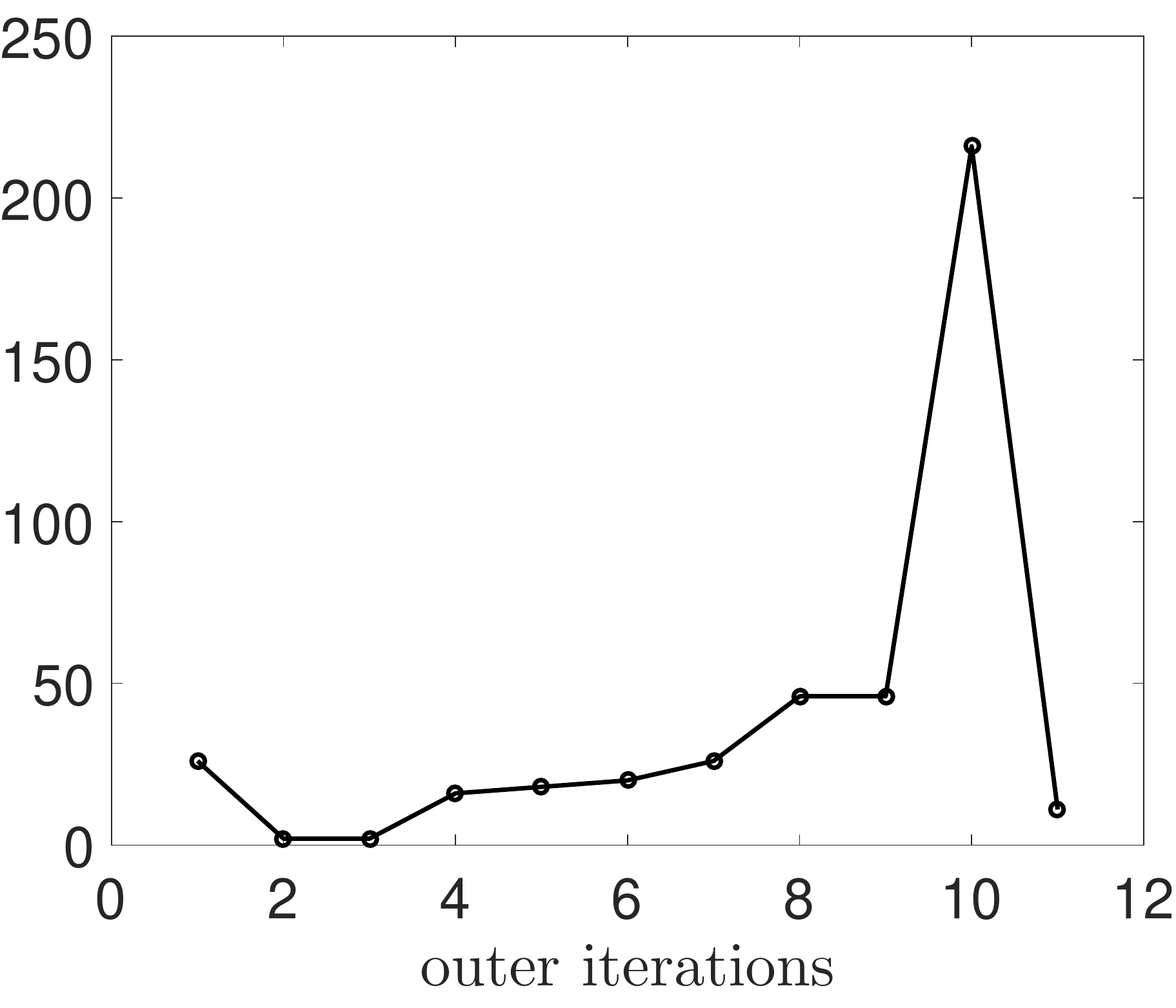}  \\
$\mathcal{A}_1$   &  $\mathcal{A}_1$  & $\mathcal{A}_1$ \\  [2mm]
\includegraphics[width=3.5cm]{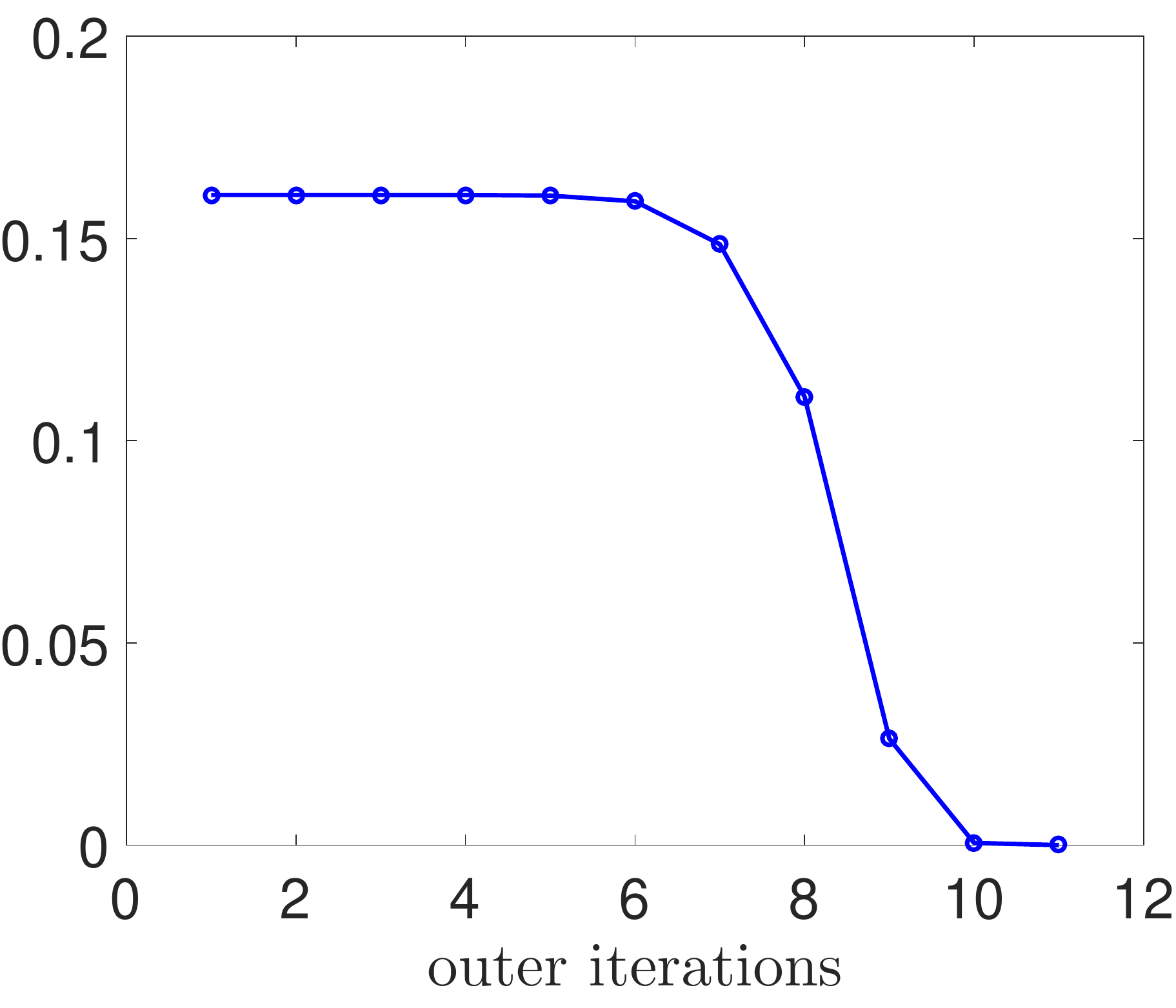} &
\includegraphics[width=3.5cm]{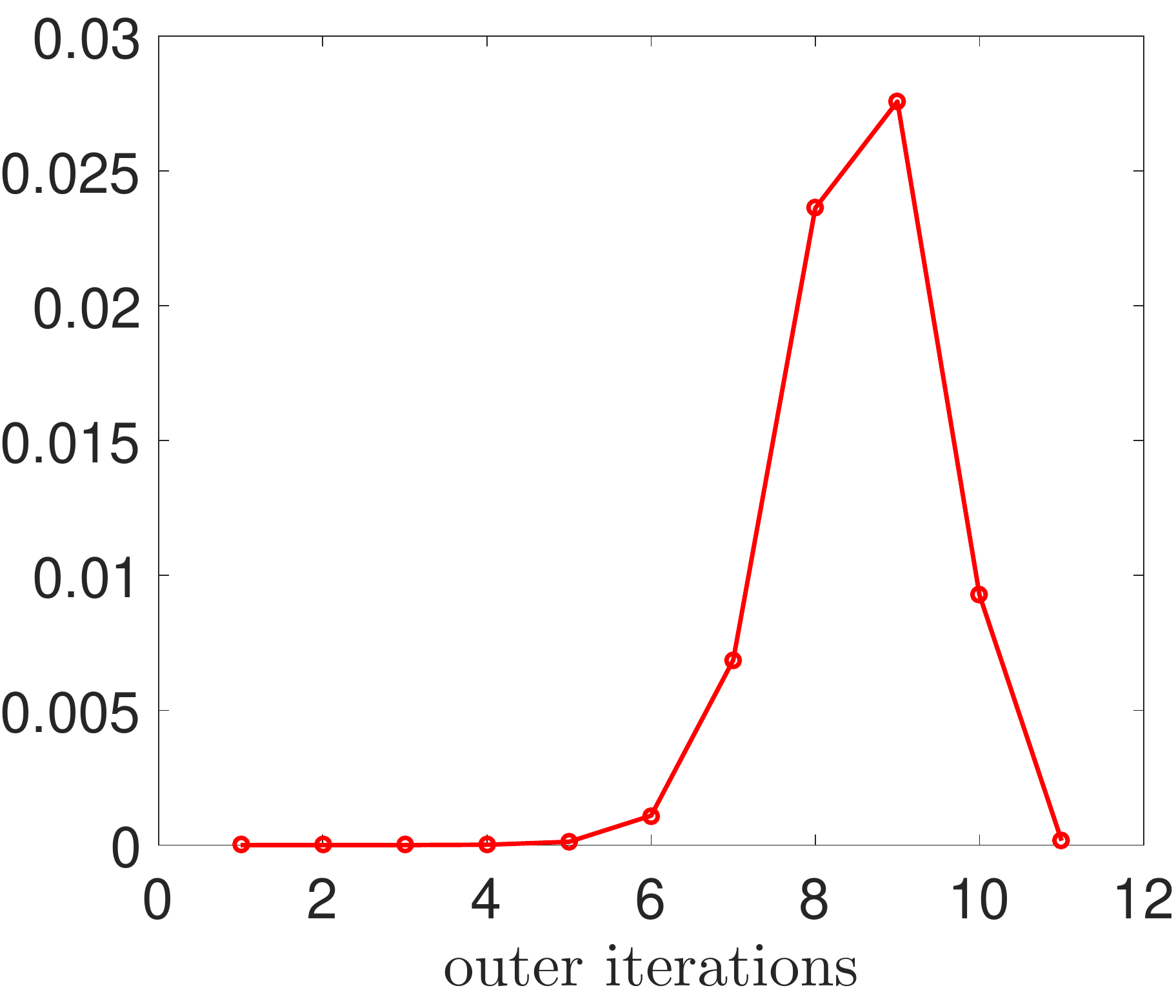} &
\includegraphics[width=3.5cm]{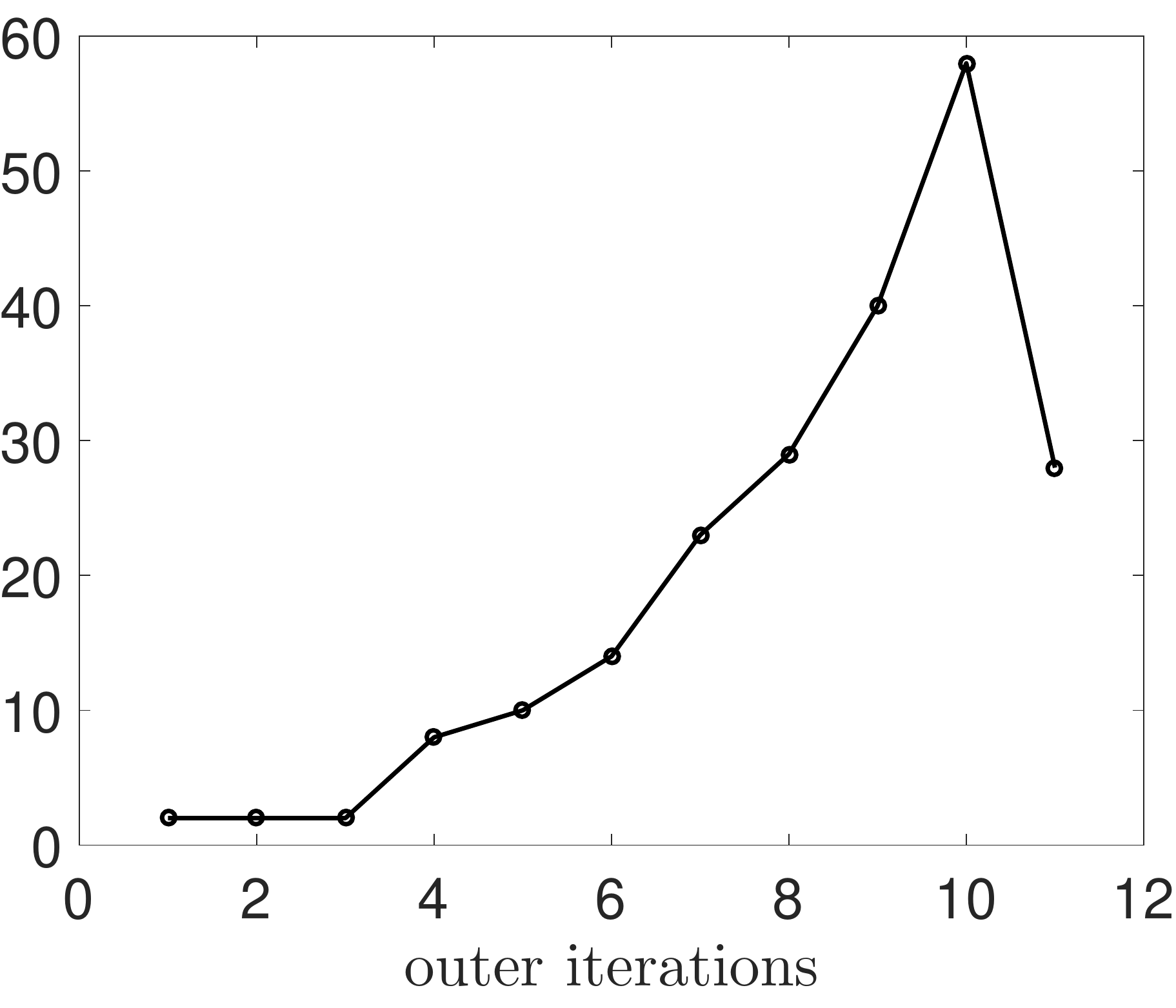}  \\
$\mathcal{A}_2$   &  $\mathcal{A}_2$  & $\mathcal{A}_2$ \\  [2mm]
\includegraphics[width=3.5cm]{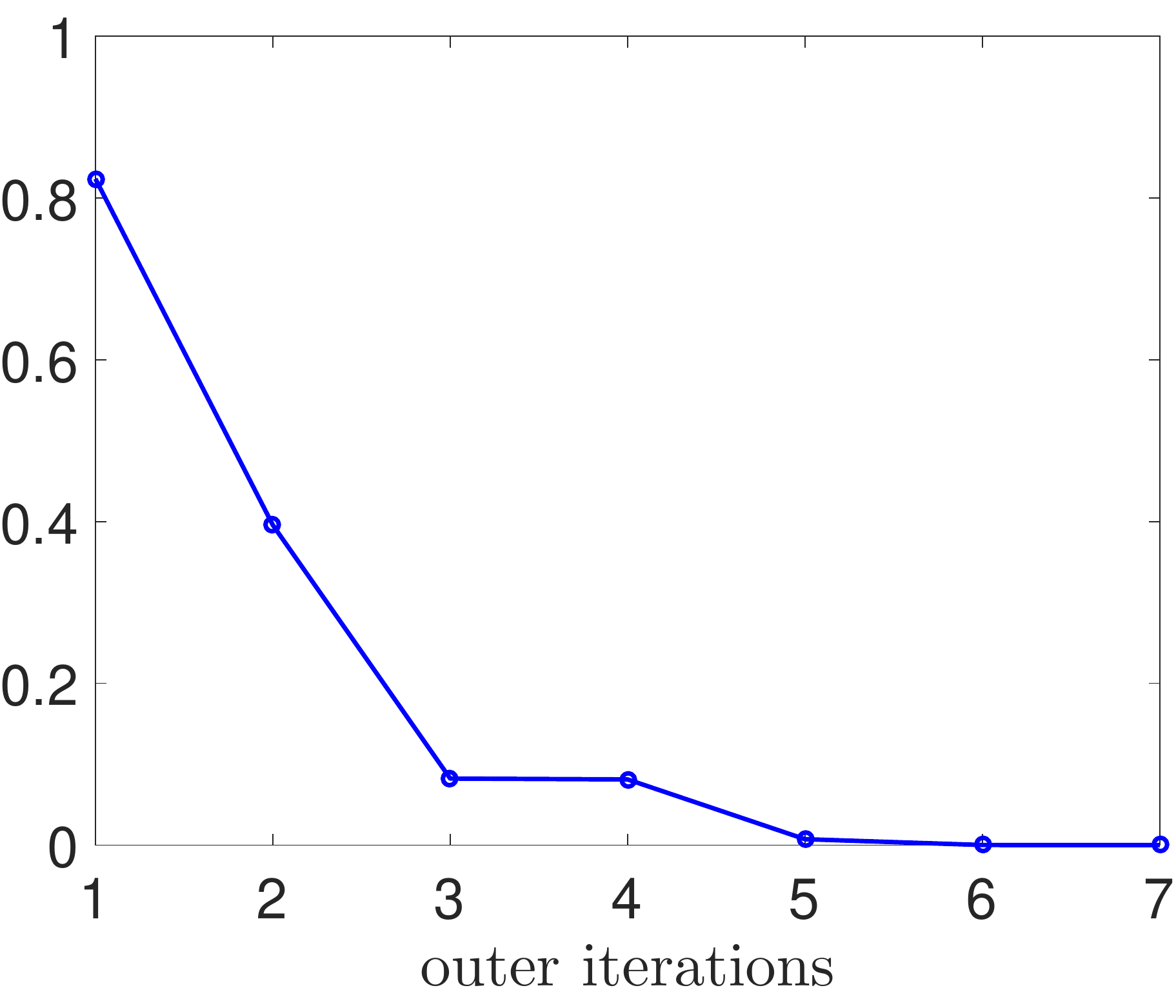} &
\includegraphics[width=3.5cm]{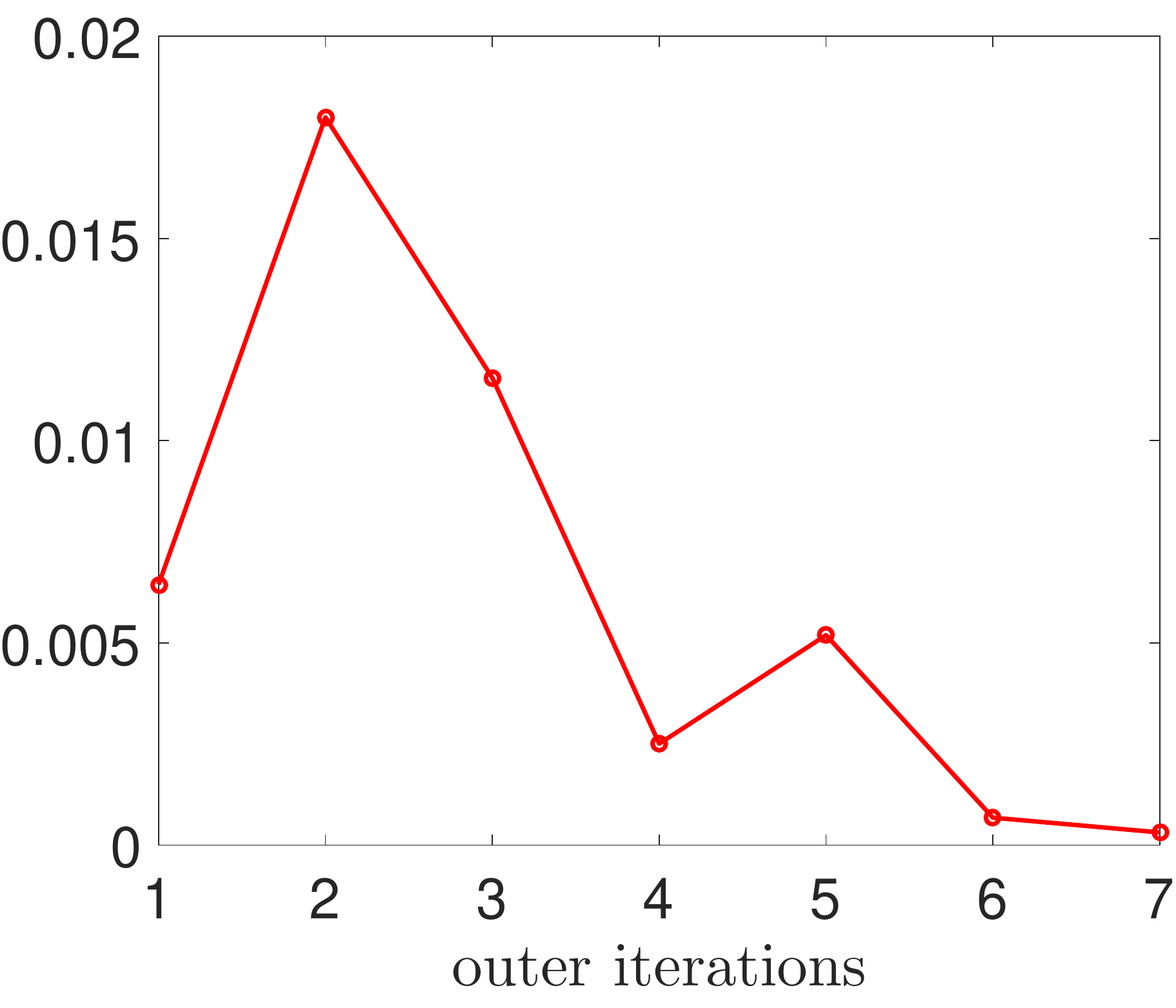} &
\includegraphics[width=3.5cm]{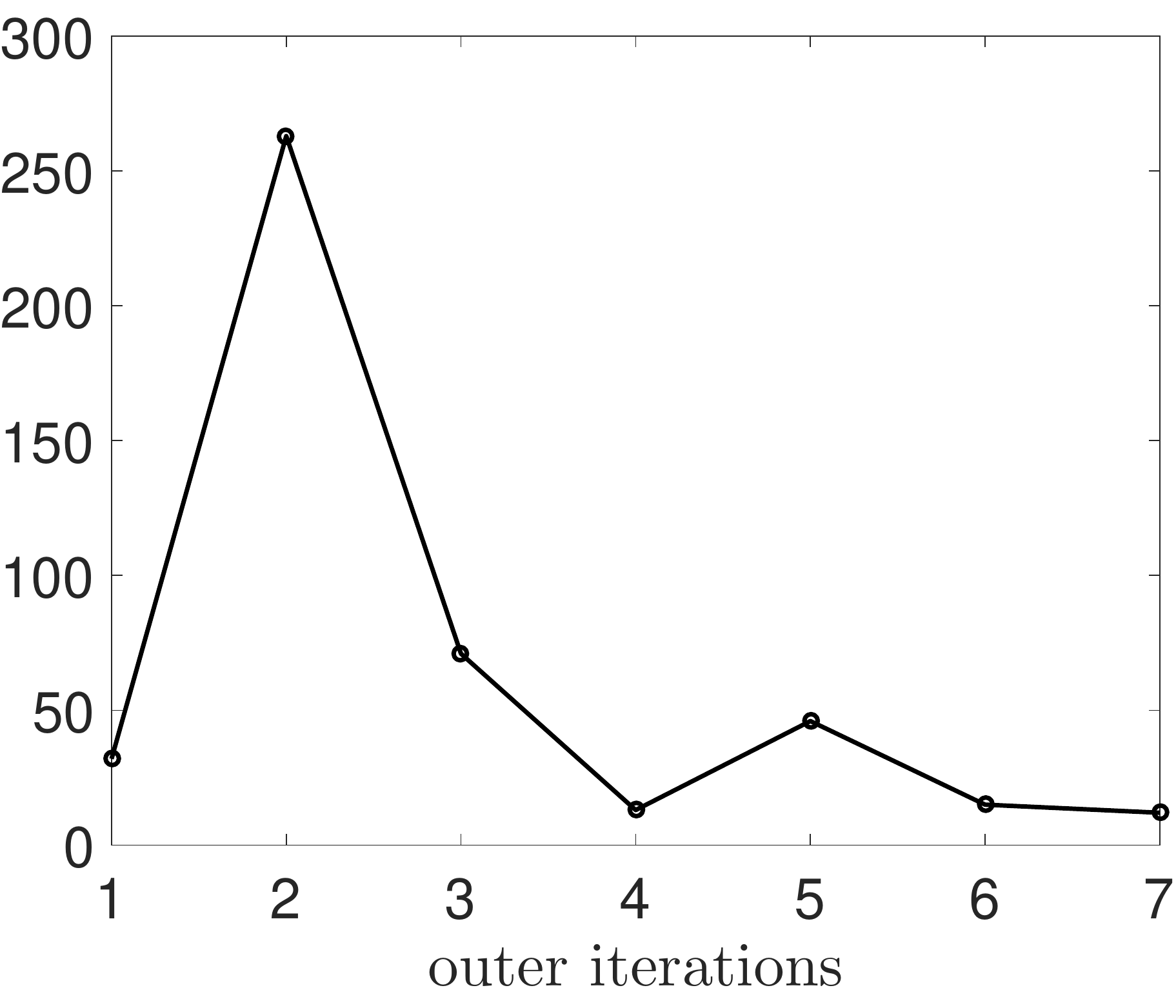}  \\
$\mathcal{A}_3$   &  $\mathcal{A}_3$  & $\mathcal{A}_3$ \\  [2mm]
\includegraphics[width=3.5cm]{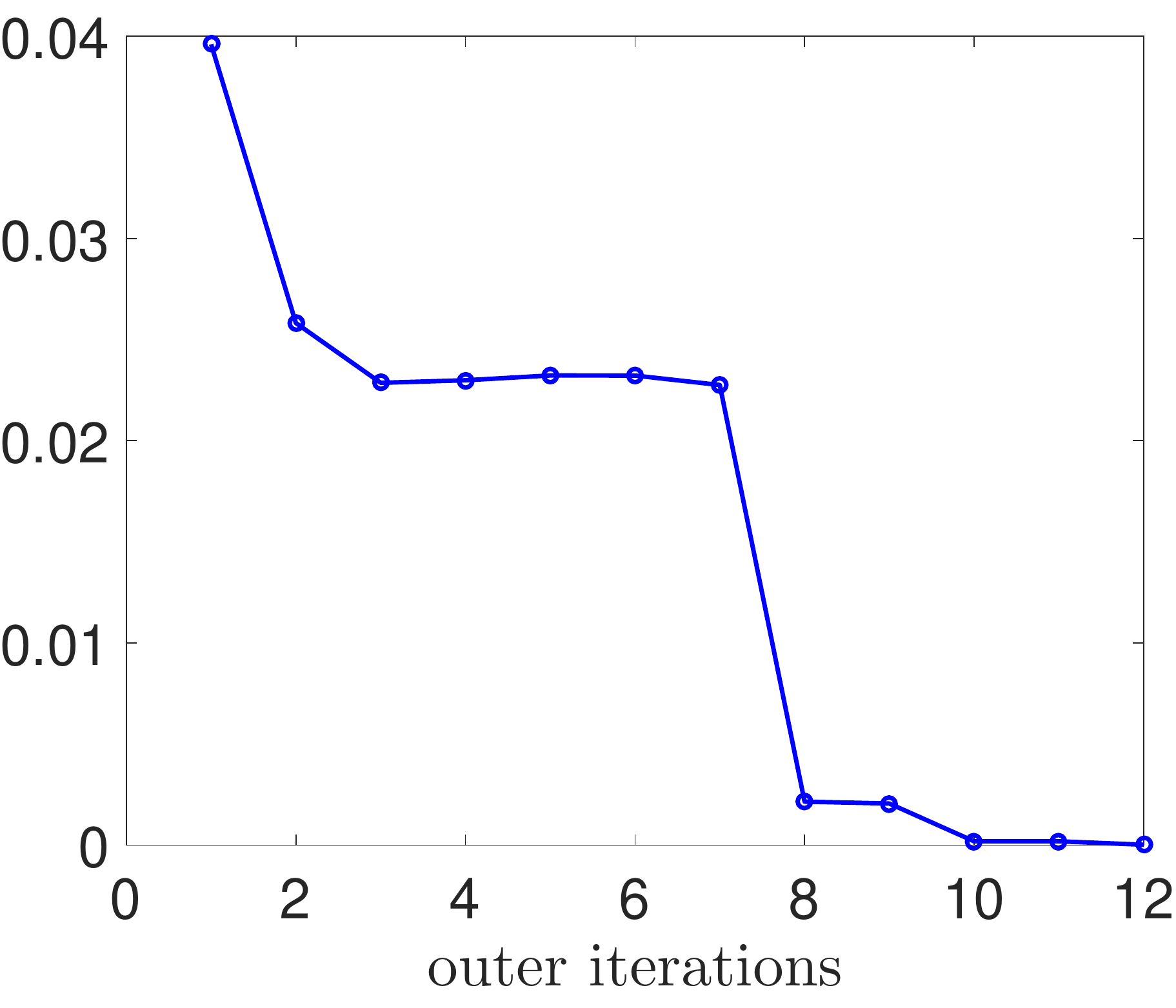} &
\includegraphics[width=3.5cm]{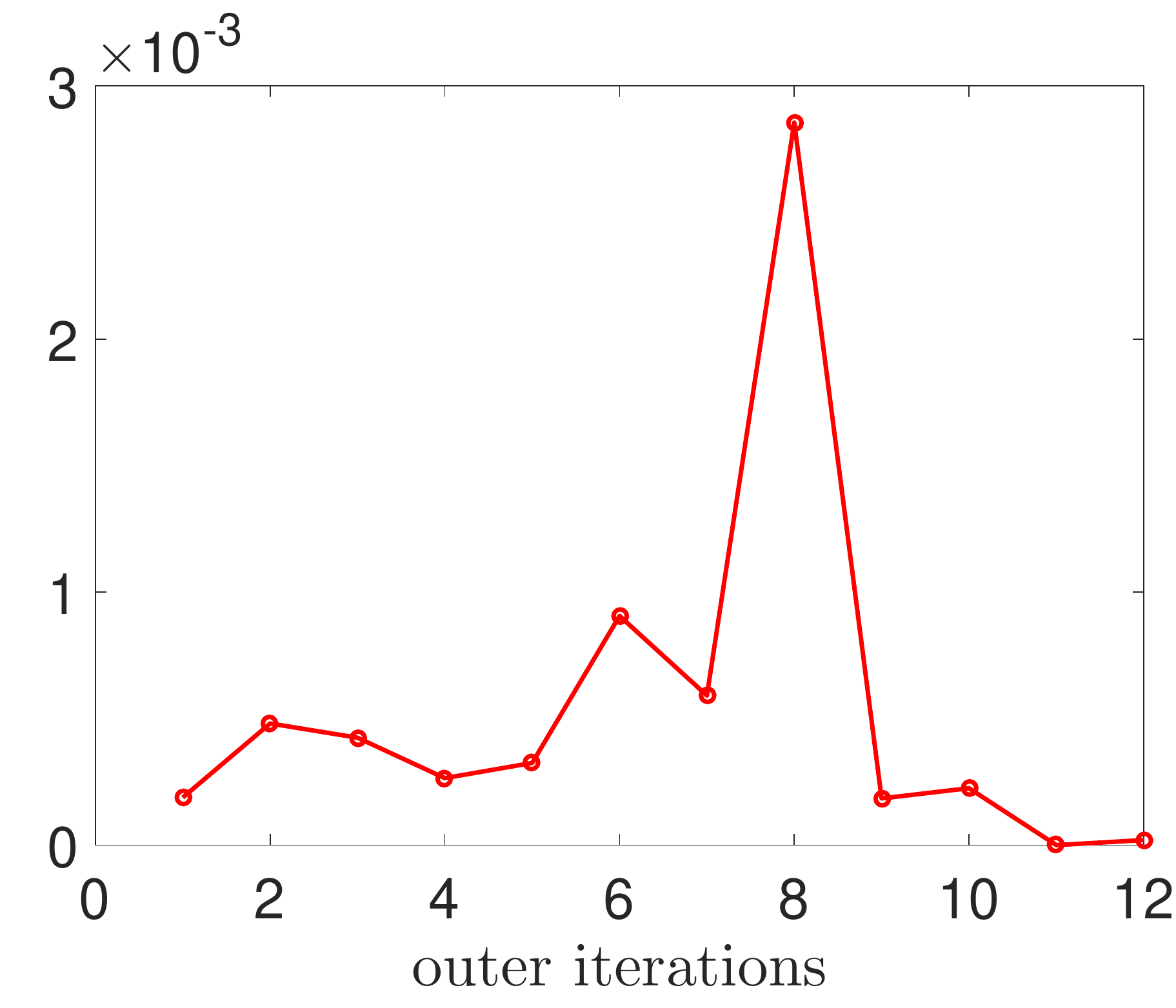} &
\includegraphics[width=3.5cm]{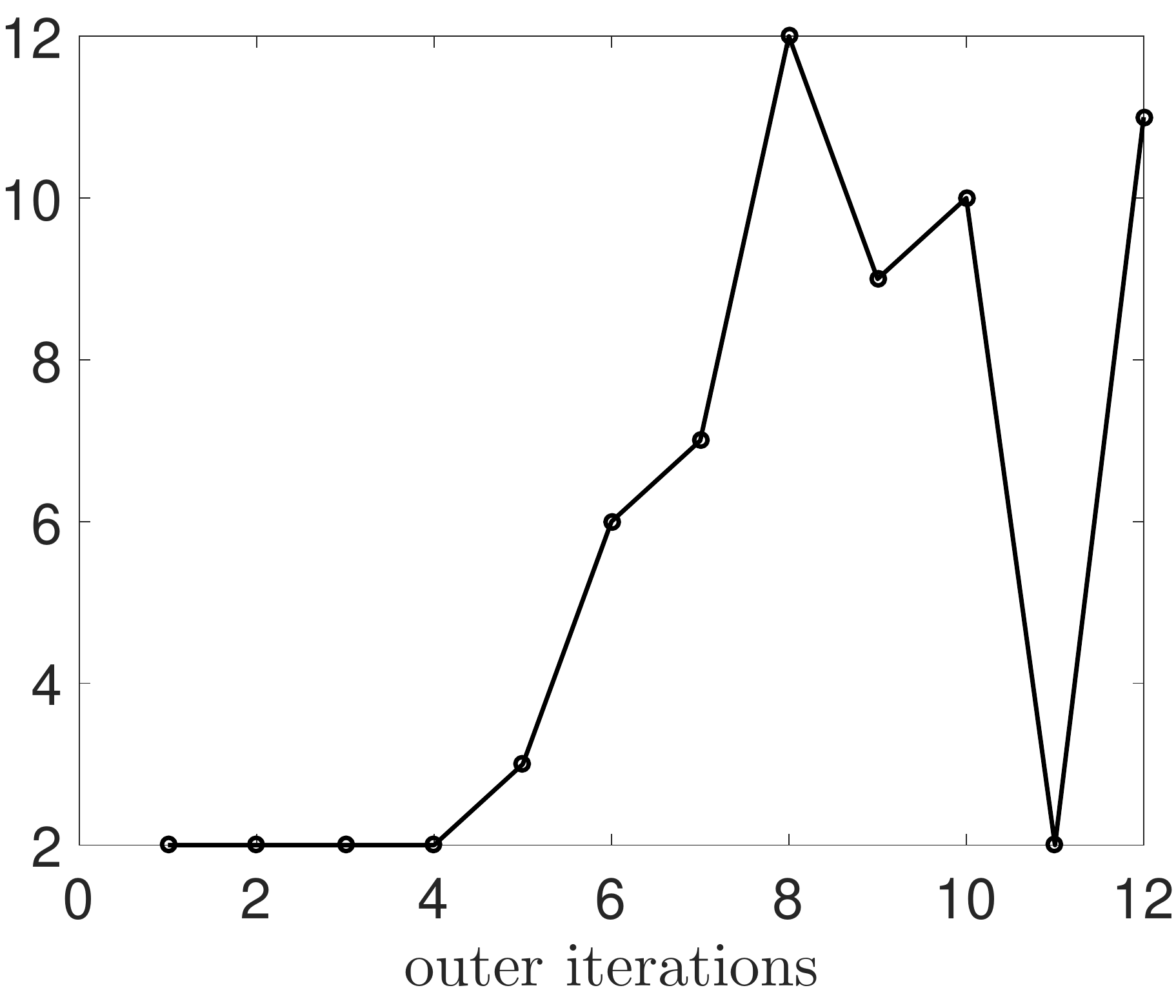}  \\
$\mathcal{A}_4$   &  $\mathcal{A}_4$  & $\mathcal{A}_4$
\end{tabular}
\caption{The convergence behaviour of OD-ALM on $\mathcal{A}_1,\ldots,\mathcal{A}_4$.
The first column is about $\theta_{[k]}$, the second column is about
$\|\mat{v}_{[k]} - \mat{v}_{[k-1]}\|/\|\mat{v}_{[k-1]}\|$, and the last column is about
the number of inner iterations. All values are shown as functions of the number of
outer iterations.}
\label{1-4}
\end{center}
\end{figure}

\begin{figure}[!ht]
\scriptsize
\begin{center}
\begin{tabular}{c@{\hskip 0.2cm}c@{\hskip 0.2cm}c}
\includegraphics[width=3.5cm]{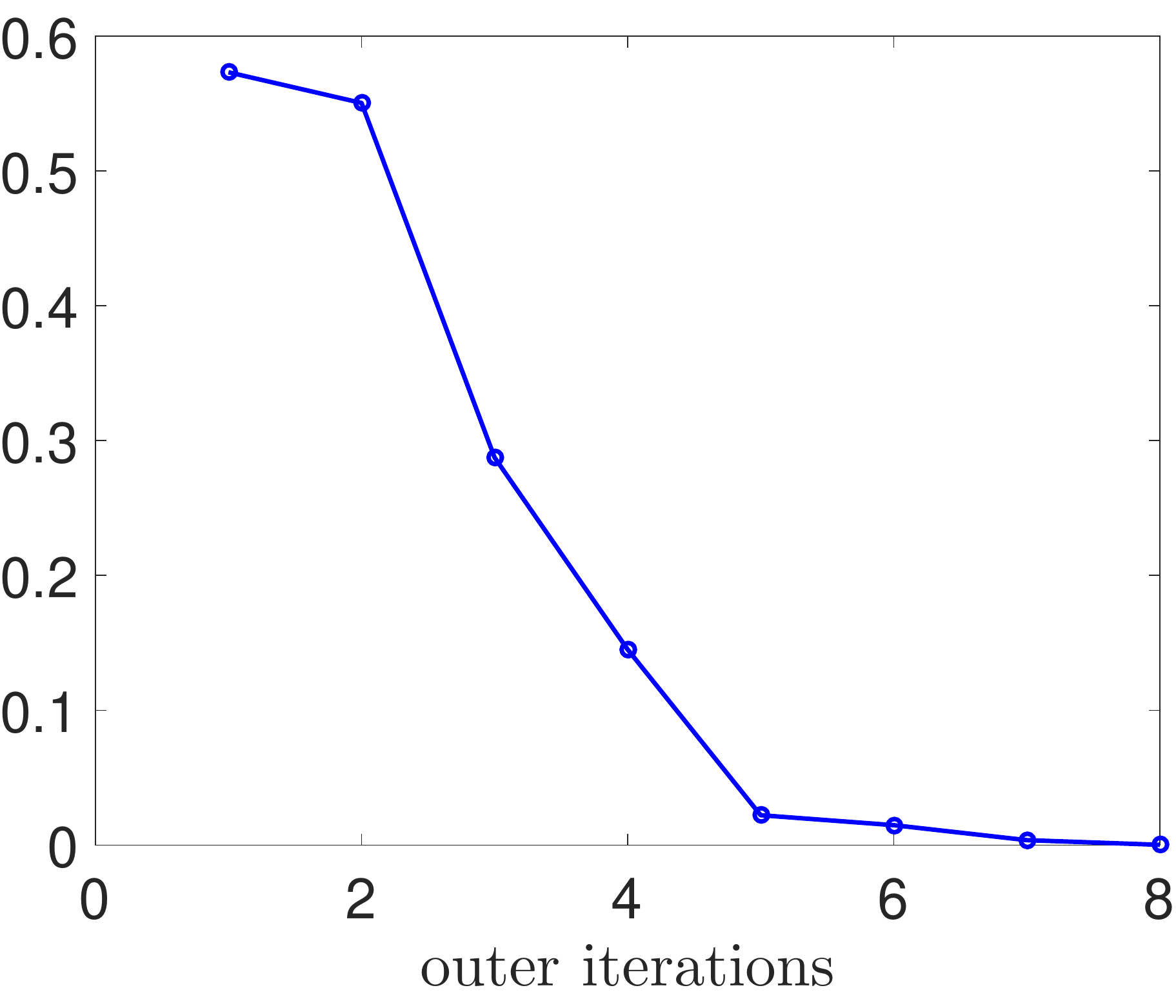} &
\includegraphics[width=3.5cm]{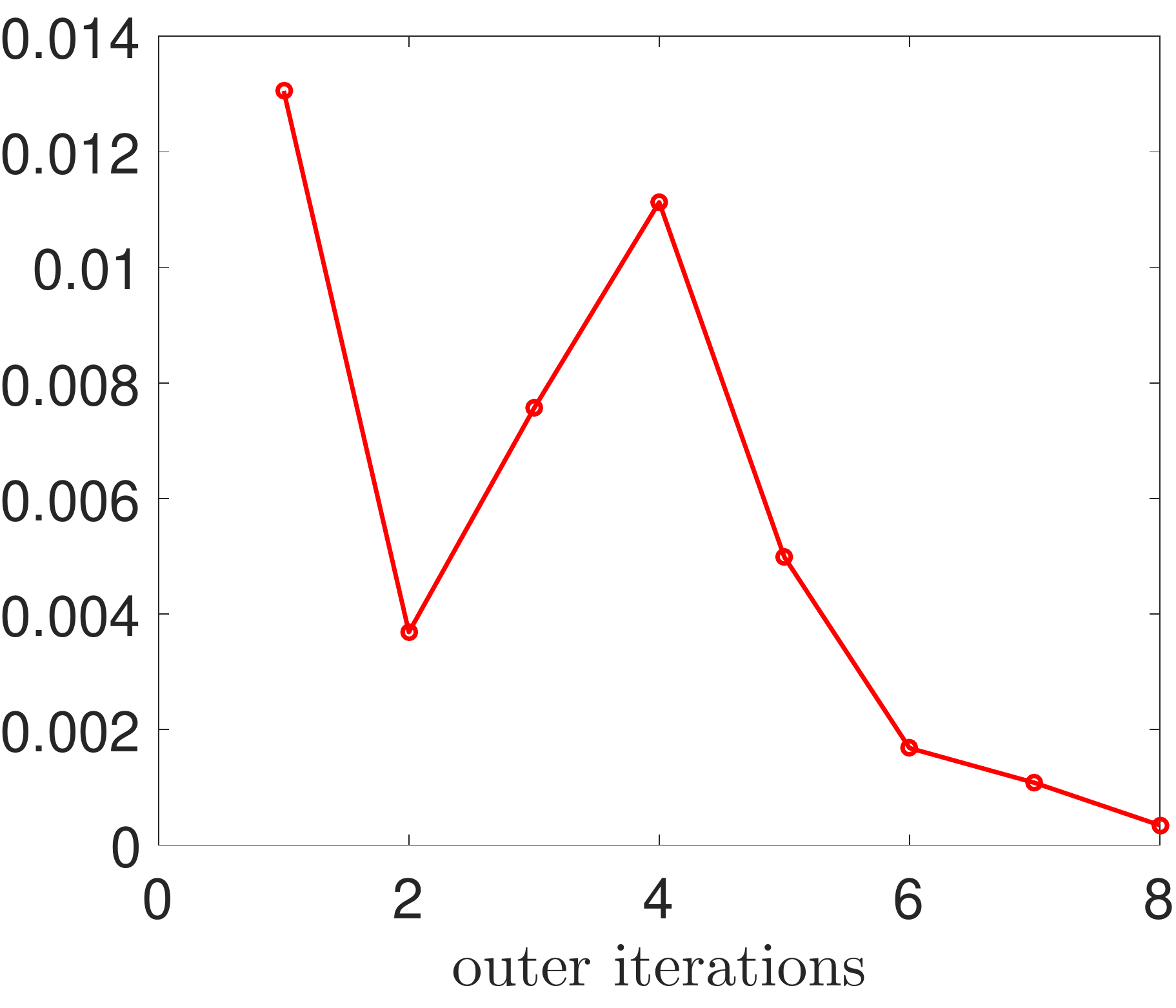} &
\includegraphics[width=3.5cm]{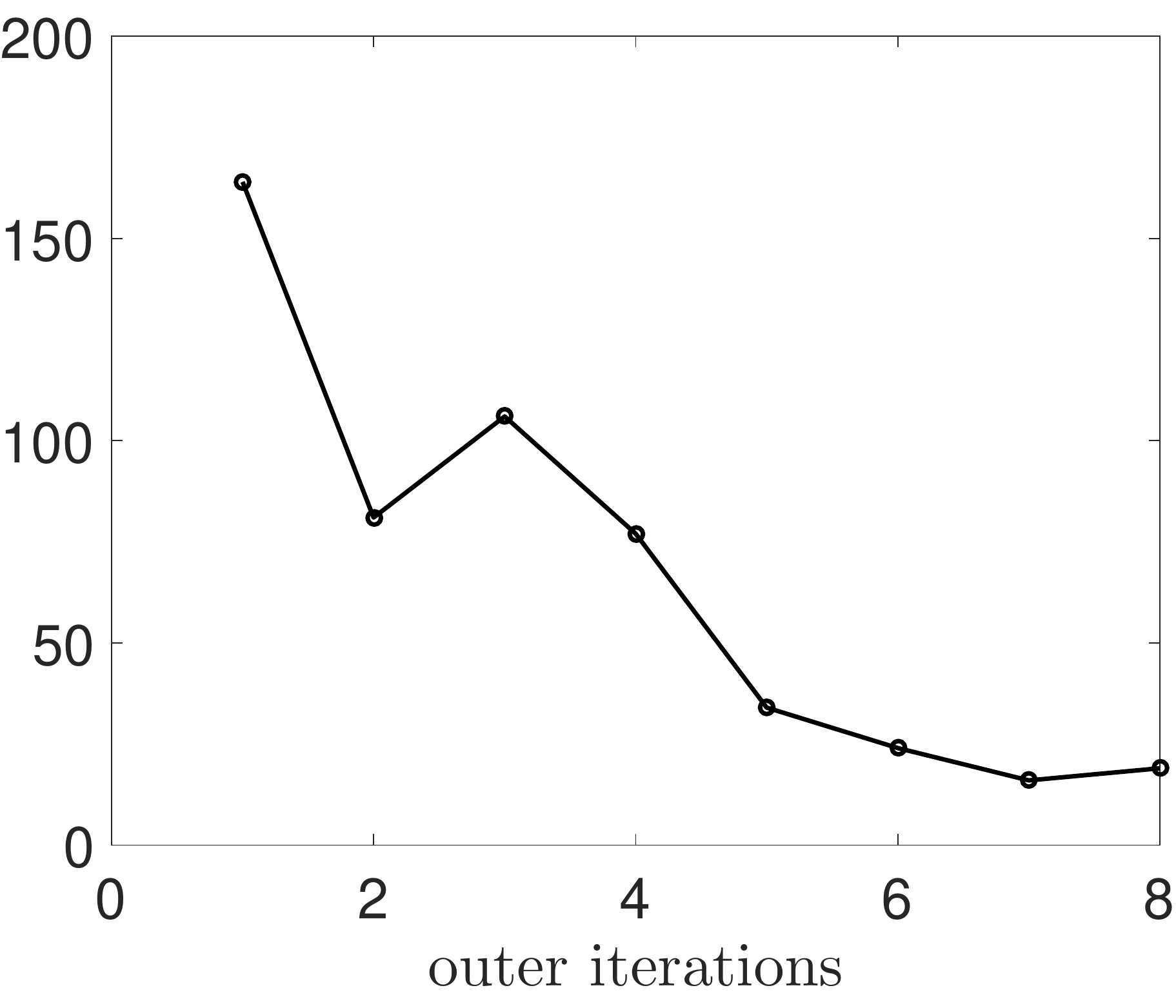}  \\
$\mathcal{A}_5$   &  $\mathcal{A}_5$  & $\mathcal{A}_5$ \\  [2mm]
\includegraphics[width=3.5cm]{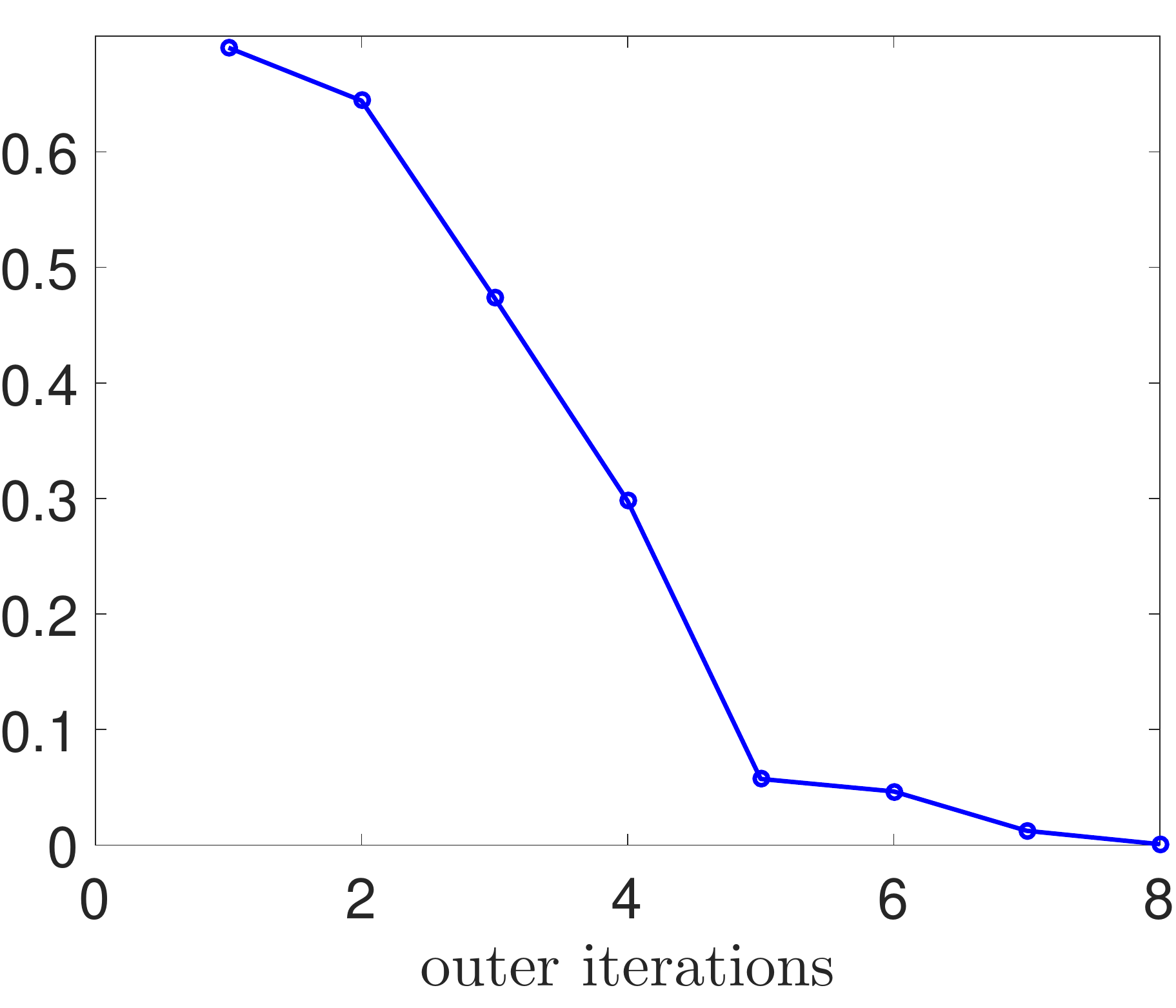} &
\includegraphics[width=3.5cm]{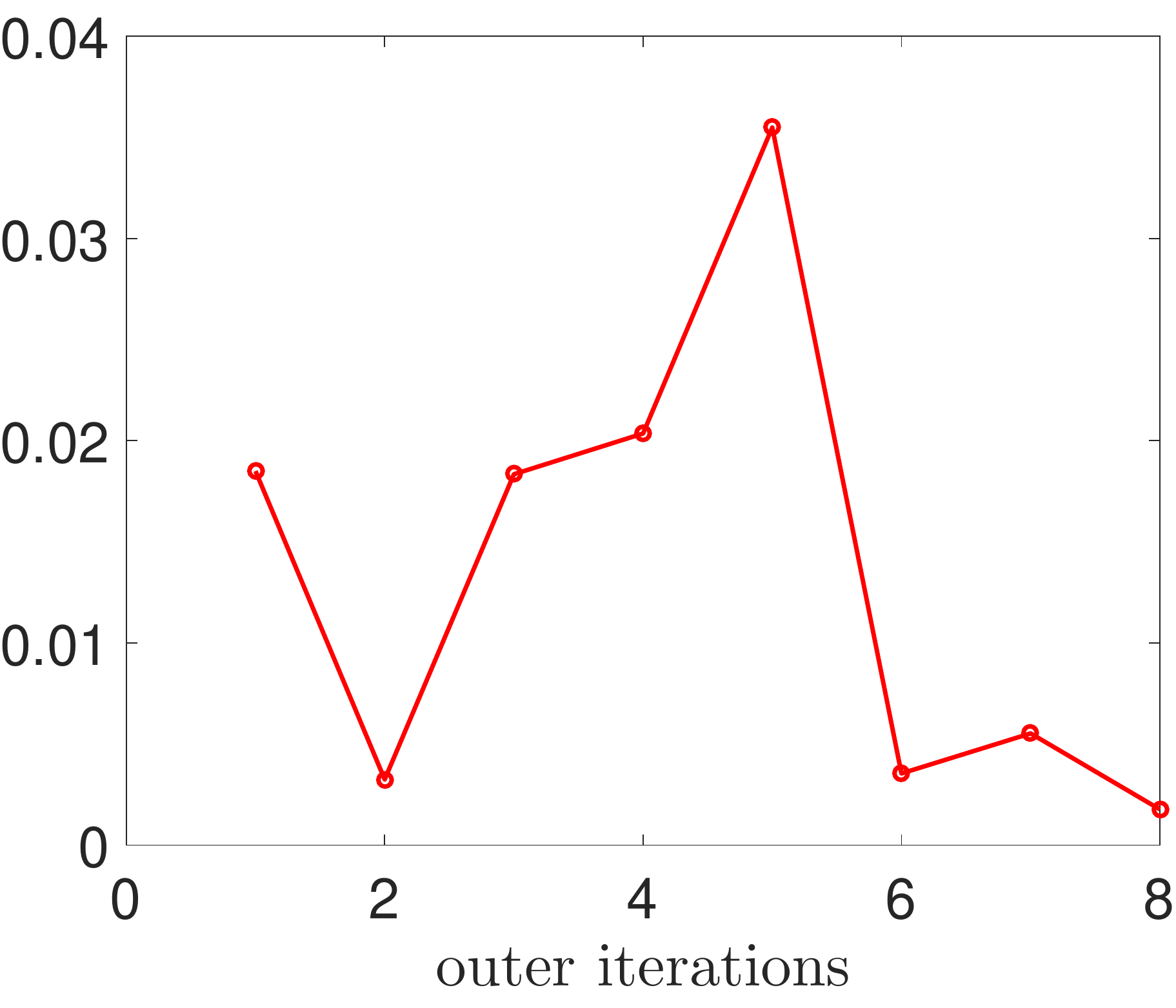} &
\includegraphics[width=3.5cm]{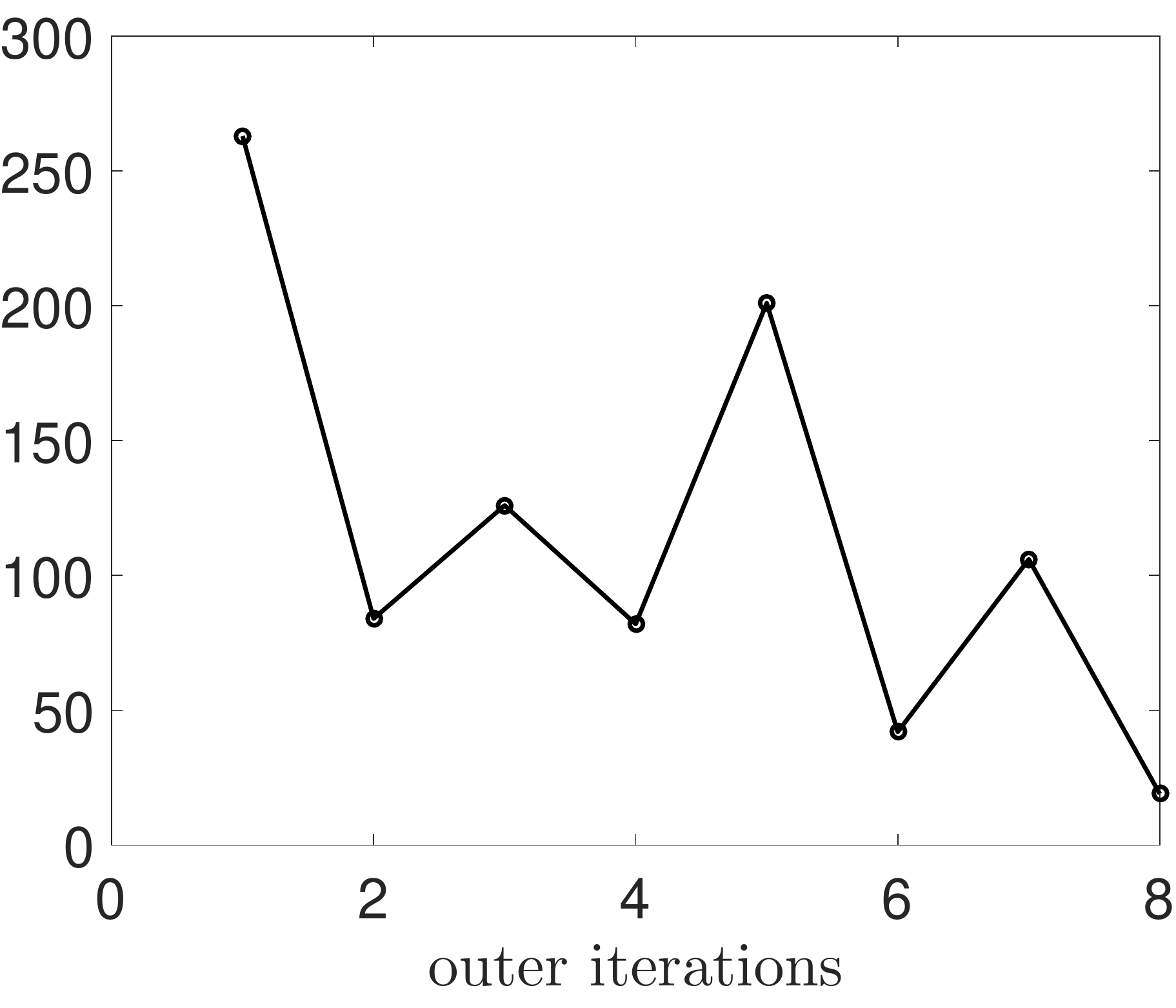}  \\
$\mathcal{A}_6$   &  $\mathcal{A}_6$  & $\mathcal{A}_6$ \\  [2mm]
\includegraphics[width=3.5cm]{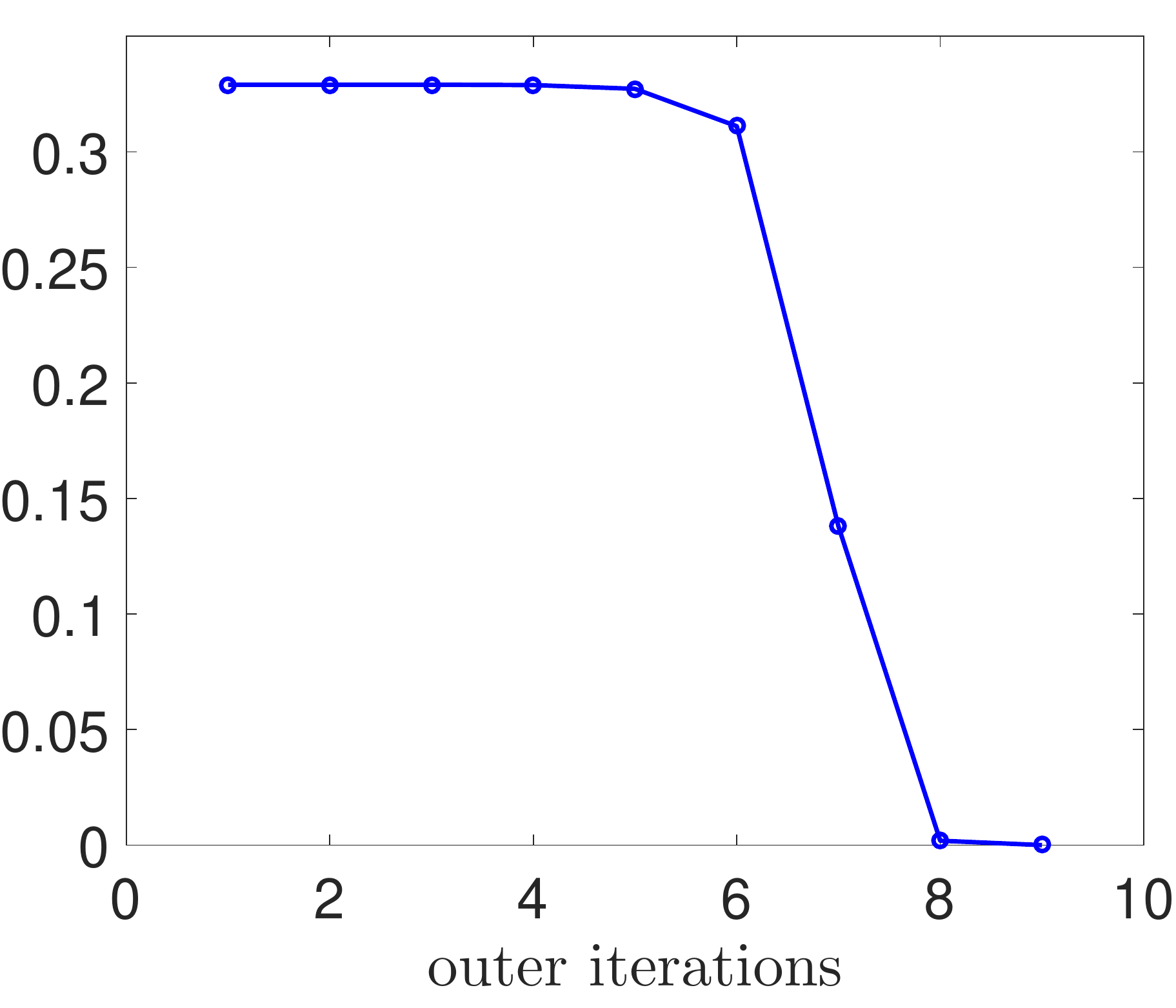} &
\includegraphics[width=3.5cm]{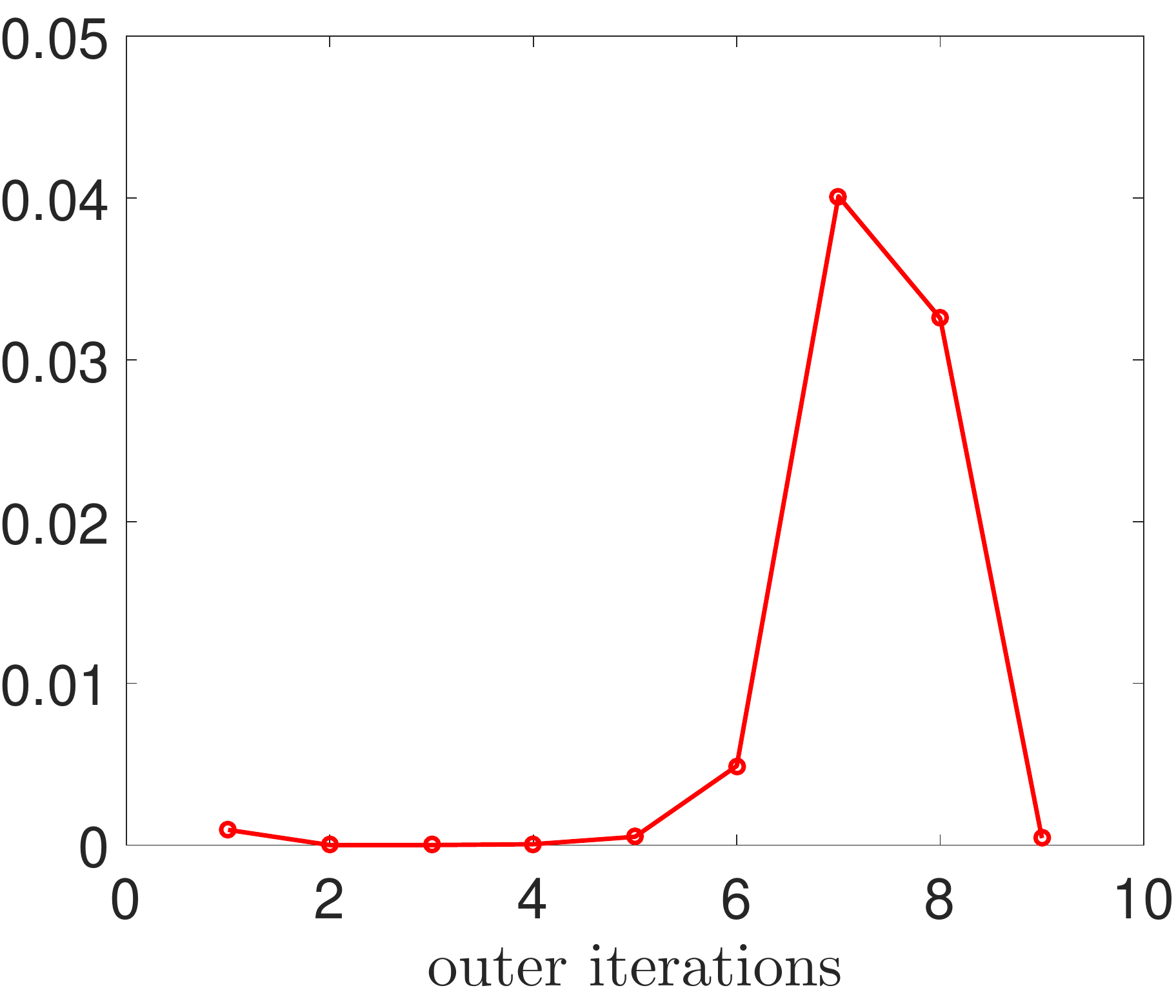} &
\includegraphics[width=3.5cm]{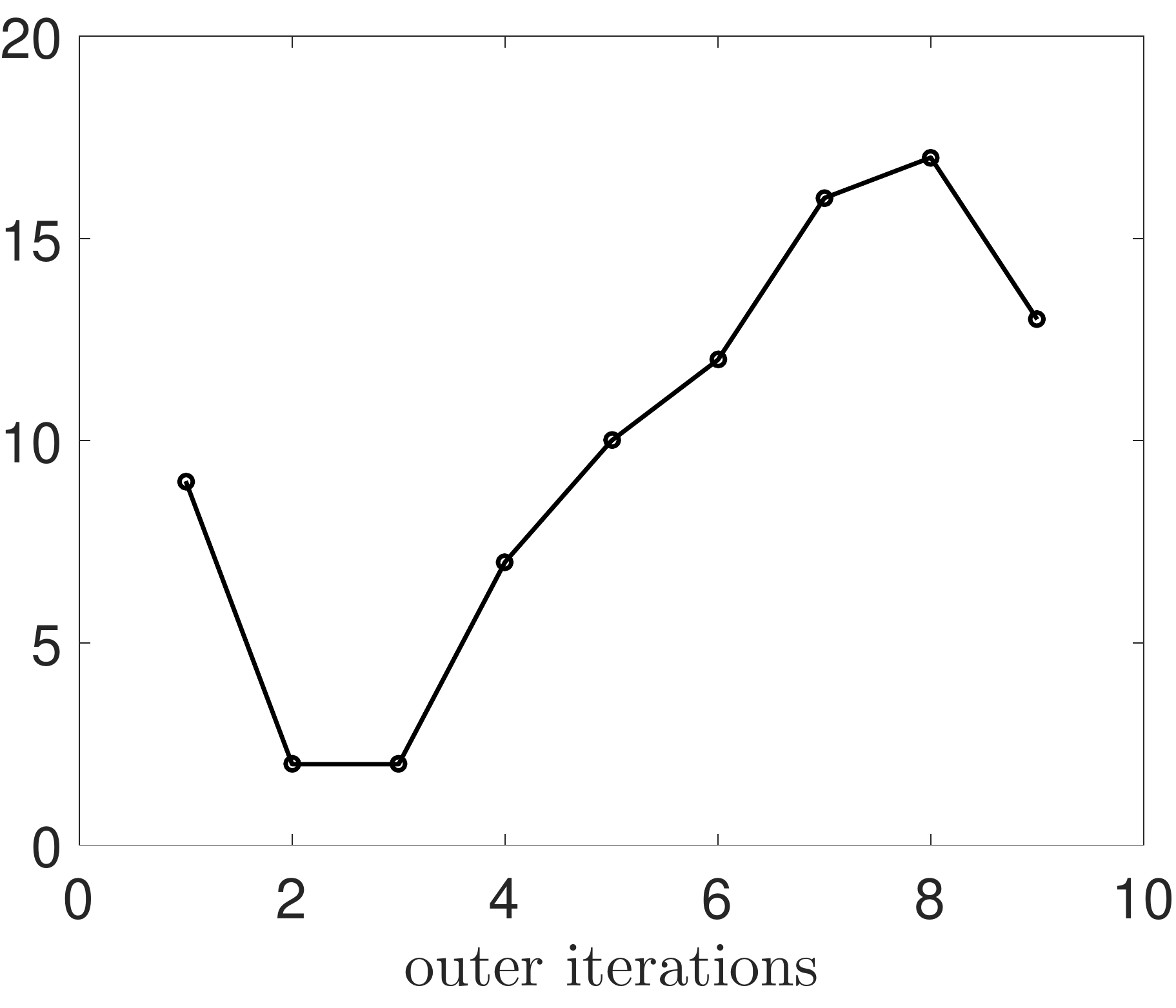}  \\
$\mathcal{A}_7$   &  $\mathcal{A}_7$  & $\mathcal{A}_7$ \\  [2mm]
\includegraphics[width=3.5cm]{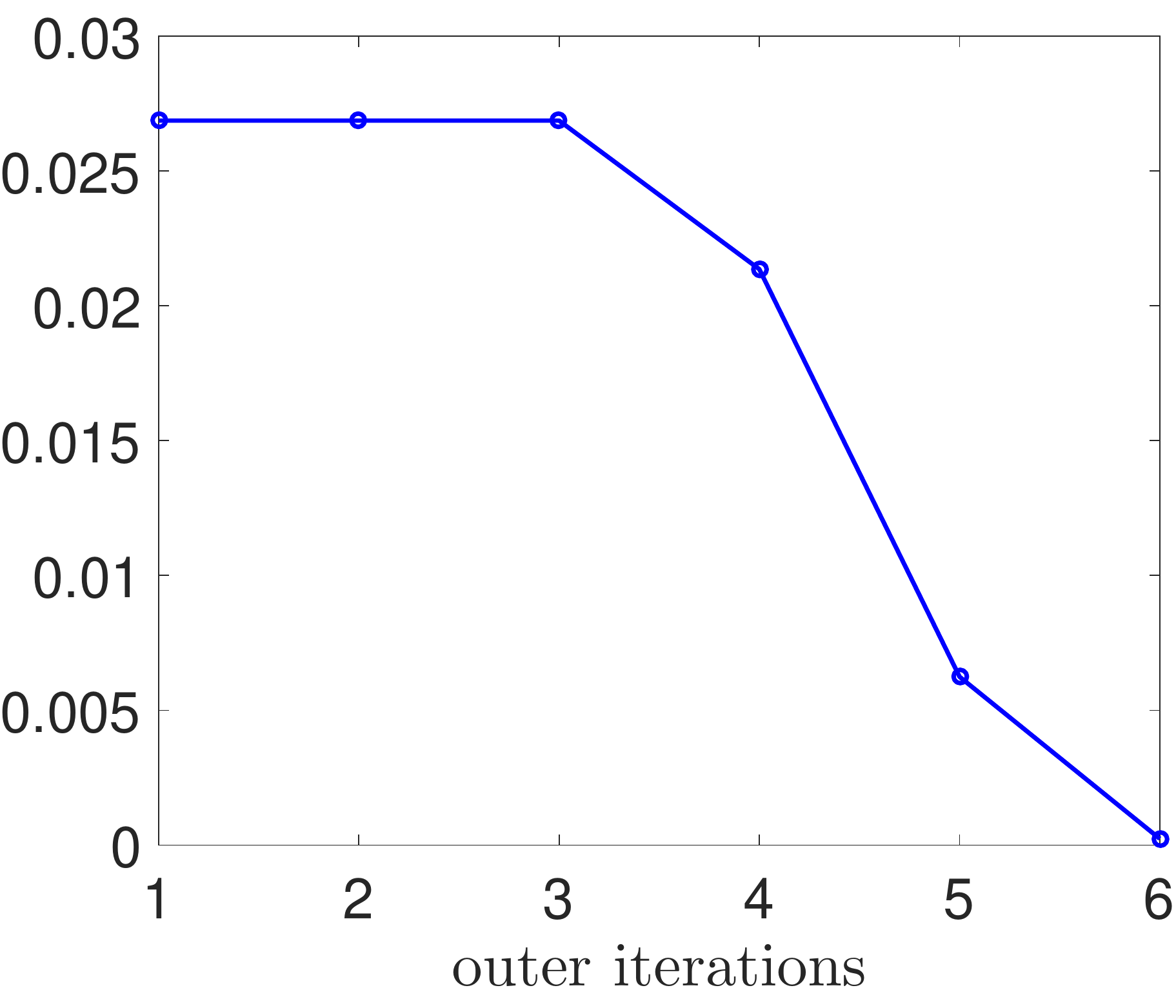} &
\includegraphics[width=3.5cm]{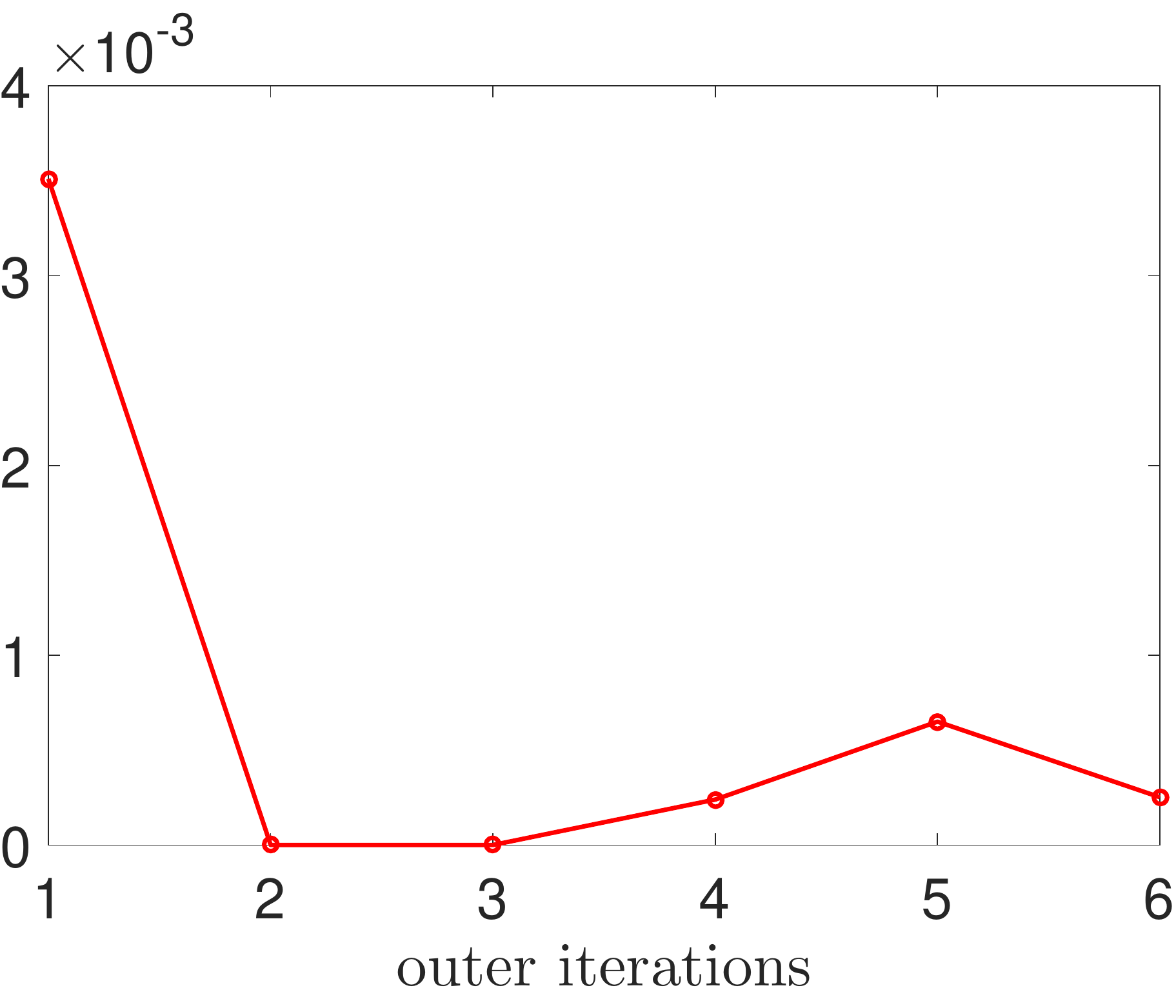} &
\includegraphics[width=3.5cm]{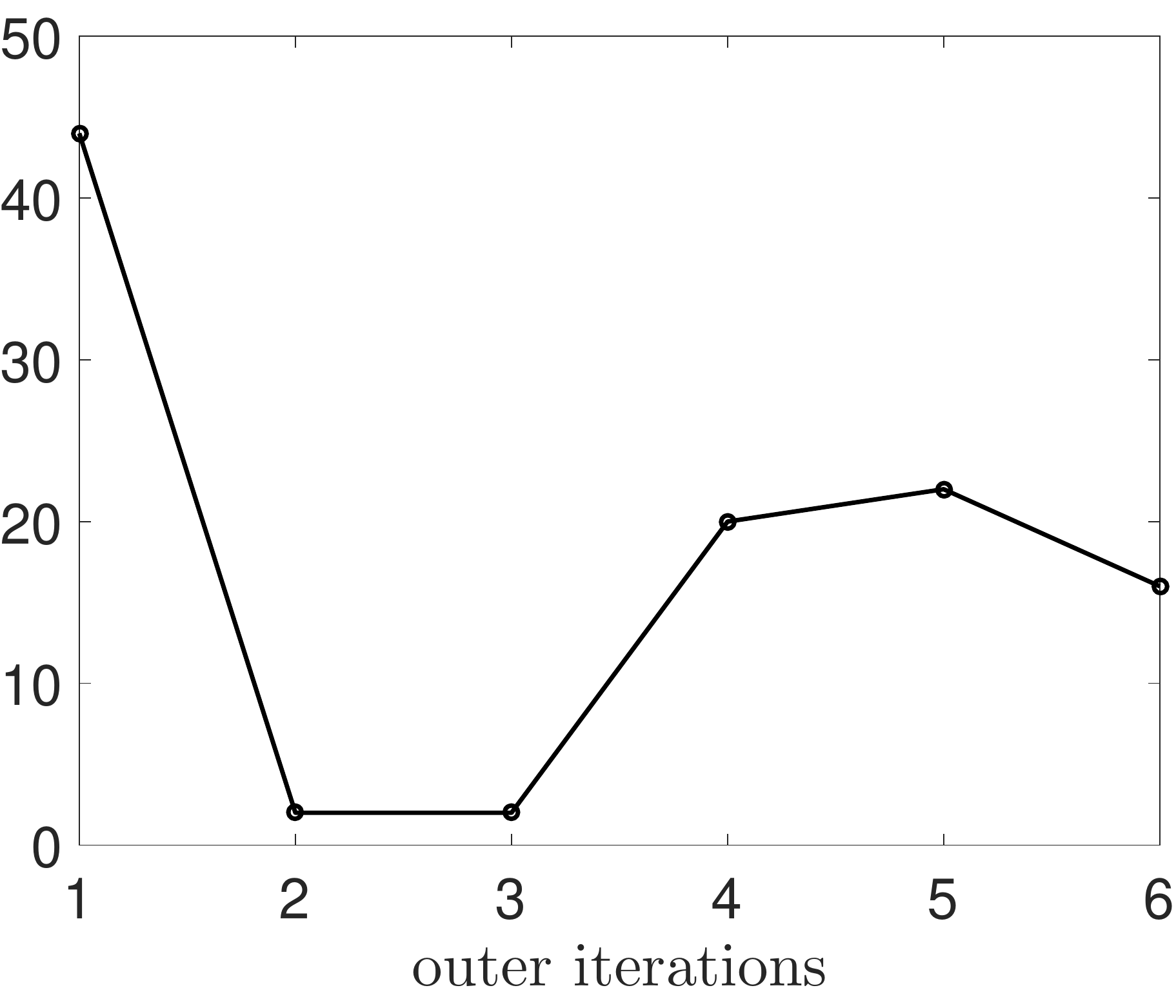}  \\
$\mathcal{A}_8$   &  $\mathcal{A}_8$  & $\mathcal{A}_8$
\end{tabular}
\caption{The convergence behaviour of OD-ALM on $\mathcal{A}_5,\ldots,\mathcal{A}_8$.
The three columns have the same meaning as in Figure \ref{1-4}.}
\label{5-8}
\end{center}
\end{figure}

The value of $\theta_{[k]}$ is decreasing as $k$ increases, but the situations differ greatly
for different tensors. For example, $\theta_{[k]}$ of $\mathcal{A}_7$ is almost unchanged
for the first five outer iterations, while $\theta_{[k]}$ of $\mathcal{A}_6$ decreases
from more than 0.6 to less than 0.1 in the first five outer iterations. Usually,
a big number of inner iterations brings a relatively big change of $\theta_{[k]}$.
For example, for $\mathcal{A}_3$, the number of inner iterations corresponding to $k=2$ is more than
250, resulting in the difference between $\theta_{[1]}$ and $\theta_{[2]}$ being more than 0.4.

The relative change between successive outer iterates can be relatively big for some tensors
even when $k$ is big, e.g., $\mathcal{A}_6$ and $\mathcal{A}_7$. This is data dependent.
In addition, the relative change is relatively small between the last two outer iterates for all cases.
The number of inner iterations reflects the relative change: A big number of
inner iterations often results in a big relative change between successive outer iterates.

\subsection{Comparison with other methods}

We compare our method with CP-ALS, the low rank orthogonal approximation of tensors (LROAT)
\cite{chen2008on} and the high-order power method for orthogonal low rank decomposition
(OLRD-HOP) \cite{wang2015orthogonal}. The method LROAT
fits an $(1,\cdots,N)$-orthogonal decomposition, and OLRD-HOP fits an $(N)$-orthogonal
decomposition. CP-ALS, LROAT and OLRD-HOP are all with the truncated HOSVD initialization.
CP-ALS terminates if the relative change in the function value is less than $10^{-8}$.
LROAT and OLRD-HOP terminate if the relative change between successive iterates is less
than $10^{-8}$. The maximum number of iterations is set to be 500 for all these three methods.
The results of the running time and the relative error are shown in Table
\ref{table_result}, which are averaged over 10 times repeated running.

\begin{table}[!ht]
\footnotesize
  \centering
  \renewcommand\tabcolsep{3.0pt}
  \caption{Comparison results of different methods. Here OD-ALM has been combined with
  Algorithm \ref{alg2}.}
  \label{table_result}
  {\renewcommand{\arraystretch}{1.2}
    \begin{tabular}{|c|c|cccccccc|}
    \hline
      & Method &  $\mathcal{A}_1$ & $\mathcal{A}_2$ & $\mathcal{A}_3$
      &  $\mathcal{A}_4$ & $\mathcal{A}_5$ & $\mathcal{A}_6$ & $\mathcal{A}_7$ & $\mathcal{A}_8$ \\
    \hline
    \hline
    \multirow{4}{*}{time} &  CP-ALS & 0.3 & 0.1 & 0.8 & 0.1 & 1.3 & 1.6 & 1.7 & 5.1 \\
    & OD-ALM & 2.7 & 1.6 & 3.3 & 0.5  & 4.8 & 13.2 & 15.8 & 15.3 \\
    & LROAT & 2.2 & 0.07 &  0.06 & 0.06 & 0.7 & 1.3 & 3.8 & 8.4 \\
    & OLRD-HOP & 0.6 & 0.07 & 1.3 & 1.3 & 2.1 & 2.5 & 1.2 & 2.9 \\
    \hline
    \multirow{4}{*}{RErr} & CP-ALS & 0.9953 & 0 & 0.0070 & 0.0993 & 0.1822 & 0.2363 & 0.2857 & 0.2278   \\
    & OD-ALM & 0.9954 &  0.0559 &  0.0227  & 0.0994 & 0.1831 & 0.2379 & 0.2931 & 0.2278 \\
    & LROAT &  0.9957 & 0.2890 &  0.1728  & 0.1640 & 0.3504 & 0.3263 & 0.4513  & 0.2530 \\
    & OLRD-HOP & 0.9954 & 0.1604 & 0.1117  & 0.1478 & 0.3333 & 0.3174 & 0.4510 & 0.2525  \\
    \hline
    \end{tabular}
    }
\end{table}

We can see that our method is much slower than the other methods. As discussed in
\cite{acar2011a}, the time cost of one outer iteration of OD-ALM is of the same order
of magnitude with CP-ALS. OD-ALM needs several outer iterations, resulting in a much longer time
cost than CP-ALS. The time costs of LROAT and OLRD-HOP are close to that of CP-ALS.

As for the relative error, CP-ALS is the best, OD-ALM is the second best, and OLRD-HOP outperforms
LROAT. This is not surprising because of the relationships among the decompositions fitted by different
methods. For $\mathcal{A}_4$ whose ground truth is an orthogonal rank-5 tensor,
the OD-ALM RErr is less than the noise level 0.1, which demonstrates the effectiveness
of our method. In addition, we can find that the difference between the CP-ALS RErr and the
OD-ALM RErr is very small for real-world tensors. For $\mathcal{A}_8$,
the results of these two methods are even the same. This suggests the potential
of orthogonal decompositions in fitting real-world tensors. The small gap between the
CP-ALS RErr and the OD-ALM RErr also indicates the effectiveness of our method in some sense.

Suppose $\mat{U}_j^{(n)}$ is the $n$th normalized factor matrix corresponding to the final
result for $\mathcal{A}_{j}$ obtained by our method. We record the results of
$\mat{U}_j^{(n)^T}\mat{U}_j^{(n)}$ for $j=3,5$ in one running:
{\tiny
\begin{align*}
& \mat{U}_3^{(1)^T}\mat{U}_3^{(1)}= & & \mat{U}_3^{(2)^T}\mat{U}_3^{(2)}= \\
& \begin{bmatrix}
1   &  0.6089  &  0.6264 &  -0.3196  &  0   \\
0.6089  &  1 &   0.9814  &  0.5454   &  0.7771 \\
0.6264  &  0.9814  &  1 &   0.4745  &  0.7039  \\
-0.3196  &  0.5454 &   0.4745 &   1 &  0.9472 \\
0  &    0.7771  &  0.7039  &  0.9472 &   1
\end{bmatrix} &
& \begin{bmatrix}
1 & 0 &   -0.1713 &  -0.9277 &  -0.8513  \\
0 &  1 & 0.9685  &  0.3720 &   0.5199  \\
-0.1713 &  0.9685  &  1 & 0.5136 &   0.6367 \\
-0.9277 &   0.3720  &  0.5136 &   1 & 0.9853 \\
-0.8513 &   0.5199 &   0.6367  &  0.9853 &   1
\end{bmatrix}   \\
& \mat{U}_3^{(3)^T}\mat{U}_3^{(3)}= & & \mat{U}_3^{(4)^T}\mat{U}_3^{(4)}= \\
& \begin{bmatrix}
1 &  0.2055 &  -0.5054 &  -0.9921 &  -0.9775 \\
0.2055  &  1 &  0.7289 &  -0.0832  &  0   \\
-0.5054 &  0.7289 &   1 &  0.6030 &   0.6618  \\
-0.9921 &  -0.0832 &   0.6030 &    1 &  0.9962 \\
-0.9775  &  0 &    0.6618  &  0.9962 &   1
\end{bmatrix} &
& \begin{bmatrix}
1 &  -1 & 0 &  0 &   -0.9996  \\
-1 &    1 &  0 &  0 &    0.9996  \\
0 &  0 &    1 &  0 &         0 \\
0 & 0 &   0 &   1 &  0 \\
-0.9996 &   0.9996   &      0   & 0 &   1
\end{bmatrix};   \\
& \mat{U}_5^{(1)^T}\mat{U}_5^{(1)}= & & \mat{U}_5^{(2)^T}\mat{U}_5^{(2)}= \\
& \begin{bmatrix}
1 &  0 &    0.7831 &  -0.4958 &  0 \\
0 & 1 & 0.0954  &  0.0805 &   -0.3413 \\
0.7831 &   0.0954  &  1 &  0 &   0 \\
-0.4958 &   0.0805 &  0 &    1 &  -0.2793 \\
0 &   -0.3413  &  0 &   -0.2793 &   1
\end{bmatrix} &
& \begin{bmatrix}
1 &  0.7868 & 0 &  0 &   -0.1452 \\
0.7868  &  1 & 0 &   0 &  0   \\
0 & 0 &    1  &   -0.6186 &  -0.0751  \\
0 & 0 &   -0.6186 &   1 &   0  \\
-0.1452  & 0 &  -0.0751 & 0 &   1
\end{bmatrix}   \\
& \mat{U}_5^{(3)^T}\mat{U}_5^{(3)}= & &   \\
& \begin{bmatrix}
1 &  0.9091 &   0.9243 &   0.9867 &  -0.9640  \\
0.9091 &   1 &   0.9992  &  0.9629 &  -0.9864  \\
0.9243  &  0.9992  &  1 &   0.9720 &  -0.9920  \\
0.9867  & 0.9629  &   0.9720  &  1 &  -0.9933  \\
-0.9640 & -0.9864 &  -0.9920  & -0.9933 &   1
\end{bmatrix}.
\end{align*}
}
We also compute $\mat{U}_j^{(n)^T}\mat{U}_j^{(n)}$ for other tensors and find that the appearance
of zeros in $\mat{U}_j^{(n)^T}\mat{U}_j^{(n)}$ has no regularity. Therefore, strongly
orthogonal decompositions cannot replace orthogonal decompositions in practical
applications in general.

\section{Conclusion}\label{sec6}

We establish several basic properties of orthogonal rank.
Orthogonal rank is different from tensor rank in many aspects. For example,
a subtensor may have a larger orthogonal rank than the
whole tensor, and orthogonal rank is lower semicontinuous.

To tackle the complicated orthogonality constraints, we employ the augmented Lagrangian method
to convert the constrained problem into an unconstrained
problem. A novel orthogonalization procedure is developed to make the final result satisfy
the orthogonality condition exactly. Numerical experiments show that the proposed method has
a great advantage over the existing methods for strongly orthogonal decompositions in terms of
the approximation error.

The main drawback of our method is the time cost. This is because the time cost of one outer iteration
of OD-ALM is of the same order of magnitude with that of CP-ALS, which is not very short, and we need
several outer iterations to obtain the final result. Although the ill-conditioning is not so severe
for the augmented Lagrangian method compared to the penalty method, preconditioning is a possible way
to speed up. For preconditioning of optimization methods for CP
decompositions, one can refer to \cite{sterck2012nonlinear,de2018nonlinearly}.
Preconditioning for OD-ALM can be studied as future work. A better strategy is to design an algorithm
with a framework different from the augmented Lagrangian method. This may need further exploration of
orthogonal decompositions.



\bibliographystyle{abbrv}
\bibliography{ref}

\end{document}